\documentclass{amsart}

\usepackage{a4wide,amsmath,amssymb}
\usepackage[toc,page]{appendix}

\numberwithin{equation}{section}

\usepackage{enumerate}

\DeclareMathOperator{\sgn}{sgn}
\DeclareMathOperator{\Id}{Id}

\DeclareMathOperator{\argmax}{argmax}

\DeclareMathOperator{\Mean}{\mathbf{E}}

\DeclareMathOperator{\M}{\mathbf{E}}
\DeclareMathOperator{\Prb}{\mathbf{P}}
\DeclareMathOperator{\Law}{Law}

\newtheorem{thm}{Theorem}[section]
\newtheorem{prop}[thm]{Proposition}

\newtheorem{cor}[thm]{Corollary}
\newtheorem{rem}{Remark}[section]
\newtheorem{de}{Definition}[section]
\newtheorem{ex}{Example}[section]

\newtheorem{defn}{Definition}[section]

\newtheorem{exmpl}{Example}[section]

\def\eps{\varepsilon}
\def\a{\alpha}
\def\b{\beta}
\def\de{\partial}
\def\d{\delta}
\def\g{\gamma}
\def\l{\lambda}
\def\s{\sigma}

  \newcommand{\N}{{\mathbb N}}
\newcommand{\Z}{{\mathbb Z}}

\newcommand{\R}{{\mathbb R}}

\newcommand{\T}{{\mathbb T}}

\begin{document}
\date{July 10, 2017; revised February 8, 2018}

\title[]{On non-uniqueness and uniqueness of solutions in finite-horizon Mean Field Games}
\author[
]{Martino Bardi 
 {and} Markus Fischer
 }
 \address{Department of Mathematics, University of Padova, Via Trieste 63, 35121 Padova, Italy} \email{bardi@math.unipd.it, fischer@math.unipd.it}
 \thanks{
The authors are partially supported by the research projects  ``Mean-Field Games and Nonlinear PDEs'' of the University of Padova, and  ``Nonlinear Partial Differential Equations: Asymptotic Problems and Mean-Field Games" of the Fondazione CaRiPaRo. They   are also members of the Gruppo Nazionale per l'Analisi Matematica, la Probabilit{\`a} e le loro Applicazioni (GNAMPA) of the Istituto Nazionale di Alta Matematica (INdAM)
}
%
\begin{abstract}
This paper presents a class of evolutive Mean Field Games with multiple solutions for all time horizons $T$ and convex but non-smooth Hamiltonian $H$, as well as for smooth $H$ and  $T$ large enough. The phenomenon is analysed in both the PDE and the probabilistic setting. The examples are compared with the current theory about uniqueness of solutions.  In particular, a new result on uniqueness for the MFG PDEs with small data, e.g., small $T$, is proved. Some results are also extended to MFGs with two populations.
\end{abstract}
\subjclass{}
\keywords{Mean Field Games, finite horizon, non-uniqueness of solutions, uniqueness of solutions, multipopulation MFG}
\maketitle
\tableofcontents

\section{Introduction}
In this paper, we study the existence of multiple solutions of Mean Field Games (briefly, MFGs) with finite horizon and of the evolutive system of PDEs associated to them, and compare the results to some uniqueness theorems. The systems of PDEs we consider are backward-forward parabolic and have the form
\begin{equation}\label{mfg-gen}
\left\{
\begin{array}{lll}
-v_t + H(Dv) = \frac 12 \s^2(x)\Delta v
+ F(x, m(t,\cdot))  \; 
 \text{ in  }(0, T)\times \R^d, 
 & v(T,x)=G(x, m(T))) ,\\ \\
m_t - \text{div}(
DH(Dv) m) = \frac 12 \Delta (\s^2(x)m)
  \quad 
\text{ in  }(0, T)\times \R^d, 
& m(0,x)= \nu (x) ,
 \end{array}
\right.\,
\end{equation}
in the unknowns $(v, m)$,  where $m(t,\cdot)$ is a probability density for each $t$, and the given running and terminal costs $F, G$ map a subset of
$\R^d\times{ \mathcal{P}}
(\R^d)$ into $\R$. In the probabilistic formulation of a Mean Field Game, the solution is, instead, a pair $(u, m)$ with $m$ as above and $u$ either an open-loop or a feedback control satisfying an optimality and a mean field condition, see the precise definitions in Section \ref{SectProbMFG}.

There are two regimes under which there is uniqueness of a classical solution to \eqref{mfg-gen}. The first and best known is the case of $H$ convex and $F, G$ increasing with respect to $m$ in the $L^1$ sense, which means that imitating the other agents is costly, see \cite{LL, Car} for the PDE proof and \cite{ahuja16, carmonaetal16} for probabilistic proofs, which also cover MFGs with a common noise. The second regime is for the length $T$ of the time horizon short enough and $H$ smooth: it was presented in a lecture of Lions on January 9th, 2009 
\cite{L}, but to our knowledge it does not appear 
 in print anywhere.

The main results in the first part of the paper give counterexamples to uniqueness when either one of these two regimes is violated. More precisely,
we describe 
an explicit 
class of problems in dimension $d=1$, with costs $F$, $G$ describing a mild preference for imitating the other agents, $H$ convex and not differentiable 
at one point, where multiple classical solutions exist for all $T>0$. We also provide a variant with smooth $H$ and $T$ larger than a certain threshold. This is done for the PDE problem \eqref{mfg-gen} in Section \ref{section pde}, and for its probabilistic formulation in Section \ref{SectProb} under somewhat weaker assumptions (e.g., the volatility $\s$ may vanish).

Explicit examples of finite horizon MFGs with multiple solutions are rare. In the probabilistic literature, based on the FBSDE approach to mean field games, examples of MFGs with multiple solutions that fall into the class of problems considered here are given in Section~5.1 of Carmona, Delarue, Lachapelle \cite{carmonaetal13} and Section~4.2 of Tchuendom \cite{tchuendom16}, where the author shows that adding a common noise restores uniqueness. The ``illuminating example'' in Lacker \cite[Section~3.3]{lacker16} 
also provides an MFG with explicitly known multiple solutions; cf.\ Example~\ref{ExmplThreeSol} below.

The second part of the paper is devoted to comparing the previous non-uniqueness examples with the assumptions of some uniqueness results. 
We argue that the monotone regime is close to being sharp, since, for instance, for costs of the form
       \begin{equation*}
   F(x,\mu)= \a x 
   \int y \,d\mu(y) + f(\mu) , \quad   G(x,\mu)= \beta x  \int y \,d\mu(y) 
    + g(\mu) ,
         \end{equation*}  
with $f, g$ 
 continuous, the costs are 
 increasing in a suitable sense and uniqueness holds if 
 \[        \a > 0 , \quad \beta \geq 0 ,
 \]
 whereas there are multiple solutions if 
  \[        \a \leq 0 , \quad \beta < 0. \]
We recall that the need of a monotonicity condition for having uniqueness for all $T>0$ and $T\to\infty$ was discussed at length in \cite{L}. 
Some  explicit counterexamples, very different from ours, were shown recently by Briani and Cardaliaguet \cite{BriCar} and Cirant and Tonon \cite{CT}. An interesting analysis  of multiple oscillating solutions via bifurcations was done very recently by Cirant \cite{Cir2018}. For stationary MFG with ergodic cost functional, examples of non-uniqueness are known since the pioneering paper of Lasry and Lions \cite{LL}, and others more explicit appear in \cite{Gu, B, BP, GNP}.

For the short horizon regime we prove a uniqueness theorem inspired by Lions \cite{L} but under weaker assumptions, different boundary and terminal conditions, and with estimates in different function spaces (see Remark \ref{plions} for more details on the differences with \cite{L}). 
It does not require any convexity of $H$ nor monotonicity of $F$ and $G$.
Our example of non-uniqueness fails to satisfy more than one assumption of this theorem, but the crucial one seems to be the smoothness of the Hamiltonian (at least $C^{1,1}$). Some remarkable points of the proof of the uniqueness theorem are the following:

\noindent - it is largely self-contained and elementary, since the main estimates are got by energy methods;

\noindent - it shows that uniqueness holds also for any $T$, provided some other data are sufficiently small, such as the Lipschitz constant of $DH$, or the  Lipschitz constants of $F$ and $D_xG$ with respect to the density $m$ (Remarks \ref{C11} and \ref{C12});

\noindent - it incorporates an a-priori estimate of $\|m(t,\cdot)\|_\infty$ that is proved in the Appendix by a probabilistic argument. 

\noindent Moreover, it appears flexible enough to be applied in different settings, e.g., Neumann boundary conditions and several populations, see \cite{BC}. 
A different proof of uniqueness for a particular economic model under a smallness assumption on a parameter is in \cite{GraBen}.
More references to papers with related results are in Remark \ref{rmk:refs}.

Finally, Section \ref{two populations} extends some of the preceding results to Mean Field Games describing two populations of homogeneous agents, where the PDE system involves two Hamilton-Jacobi-Bellman and two Kolmogorov-Fokker-Planck equations instead of one.  In this case the monotonicity conditions are more restrictive: in addition to a form of aversion to crowd in each population, they require that the costs of intraspecific interactions are larger than the costs of the interactions between the two populations (cfr. \cite{Cir}). Therefore here there is more distance between the sufficient conditions for uniqueness and the examples of multiple solutions. We refer to \cite{BF, ABC} for motivations to the study of multipopulations MFGs, to \cite{Cir, ABC} for examples of non-uniqueness in the case of ergodic costs and stationary PDEs, and to \cite{BC} for uniqueness when the horizon is short.

We end this introduction with some additional bibliographical remarks. The theory of MFGs started with the independent work of Lasry and Lions \cite{LL} and Huang, Caines, and Malhame \cite{HCM:07ieee, HCM:06, HCM:07jssc}. It aims at modeling at a macroscopic level non-cooperative stochastic differential games with a very large number $N$  of identical players.  The rigorous justification of the PDE 
\eqref{mfg-gen} 
as limit of the systems of Bellman equations for such games as $N\to\infty$  was proved recently by Cardaliaguet, Delarue, Lasry, Lions \cite{CDLL} using in a crucial way the convexity and monotonicity conditions leading to uniqueness for  \eqref{mfg-gen}. General presentations of the field are \cite{Car, GS, GomesBook}, and many applications are analysed 
in \cite{GLL, GomesBookEcon}.
For the probabilistic approach to MFGs, in addition to the above mentioned works \cite{carmonaetal13, ahuja16, carmonaetal16, lacker16, tchuendom16}, we refer to Carmona and Delarue \cite{carmonadelarue13}, Lacker \cite{lacker15}, and Fischer \cite{Fish}.

\section{Multiple solutions of the MFG system of PDEs}
\label{section pde}
In this section we consider the backward-forward system of parabolic PDEs
\begin{equation}\label{pde}
\left\{
\begin{array}{lll}
-v_t + H(v_x) = \frac 12 \s^2(x) v_{xx} +F(x, m(t,\cdot)) , 
&
 \text{ in  }(0, T)\times \R, 
\\  v(T,x)=G(x, m(T)) ,\\ \\
m_t - (
H' (v_x) m)_x = \frac 12 (\s^2(x) m)_{xx} 
\ & 
\text{ in  }(0, T)\times \R,
\\ m(0,x)=\nu(x) ,  \quad m>0 ,  \quad \int_\R m(t,x) dx = 1 \quad\forall t>0 .
 \end{array}
\right.\,
\end{equation}
We will look for solutions such that $m(t, \cdot)$ is the density of a probability measure with finite first moment, that we denote as follows
\[
\tilde{ \mathcal{P}}_{1}(\R) := \{\mu\in L^\infty(\R) : \mu\geq 0,\; \int_\R \mu(x) dx = 1, \;  \int_{\R} |y|\, \mu(y) \,dy < + \infty\},
\]
and assume that the initial density $\nu$ is in $\tilde{ \mathcal{P}}_{1}(\R)$.
For such functions we denote 
 the mean with
$$ M(\mu):= \int_{\R} y\, \mu(y) \,dy .
$$
We recall that the Monge-Kantorovich distance between probability measures is
\[
d_1(\mu,\nu)= \sup \left\{\int_\R \phi(y) (\mu - \nu)(y) \,dy : \phi : \R\to\R \; 1-\text{Lipschitz} \right\}
\] 
and that $M(\cdot)$ is continuous for such distance.
The running and terminal costs
\[
F : \R\times \tilde{ \mathcal{P}}_{1}(\R) \to \R , \quad G : \R\times \tilde{ \mathcal{P}}_{1}(\R) \to \R
\]
satisfy the following regularity conditions 
\begin{description}
\item [(F1)] 
$\mu\mapsto F(x,\mu)$ is continuous for the distance $d_1$ locally uniformly in $x$; for all $\mu$ the function $x\mapsto F(x,\mu)$ is differentiable  with derivative denoted by $D_xF(\cdot, \mu)$, and for some $\a\in (0,1], k\in \N,$ there is $C_R$ such that
\[
|F(x,\mu) - F(y,\mu)|\leq C_R|x-y|^\a ,  \quad |D_xF(x,\mu) - D_xF(y,\mu)|\leq C_R|x-y|^\a
\]
for all $|x|, |y| \leq R$, $M(\mu)\leq R$, and
\[
\quad |F(x,\mu)|\leq C_R(1+|x|^k), \quad \forall\, x\in\R , M(\mu)\leq R ;
\]
\item [(G1)]  for all $\mu$ $x\mapsto G(x,\mu)$ has at most polynomial growth and it is differentiable  with continuous derivative denoted by $D_xG(\cdot, \mu)$. 
\end{description}
The main qualitative assumption on the costs  that allows us to build multiple solutions is the following.
\begin{description}
\item [(FG2)] 
for all $x\in\R$ and $\mu\in\tilde{ \mathcal{P}}_{1}(\R)$
\begin{equation}
\label{sign}
M(\mu) D_xF(x,\mu) \leq 0 , \quad M(\mu) D_xG(x,\mu) \leq 0 ,
\end{equation}
and for $M(\mu)\ne0$ either  $D_xF(\cdot, \mu)\not\equiv 0$ or  $D_xG(\cdot, \mu)\not\equiv 0$.
\end{description}
The meaning of this 
condition is that it is less costly to move to the right if $M(\mu)>0$ and to move to the left if $M(\mu)<0$, so in some sense it is rewarding for a representative agent to imitate the behaviour of the entire population. This is consistent with the known fact that aversion to crowd is related to the monotonicity conditions of Lasry and Lions that imply uniqueness of the solution \cite{LL, GLL}.
\begin{ex}
\label{ese1}
\upshape
Consider a running cost $F$ of the form
\[
F(x, \mu)= f_1\left(x, \int_{\R^d} k
(x,y) \mu(y) dy\right) f_2(M(\mu)) + f_3(\mu) 
\]
with $f_1, k_1\in C_1(\R^2)$ Lipschitz with Lipschitz derivatives, $f_2\in C(\R)$,  and $f_3$ $d_1$-continuous. Then (F1) holds. Next assume
\[
rf_2(r) \geq 0  \quad \forall\, r\in\R, \quad f_2(r)\ne 0 \quad\forall\, r\ne 0 ,
\]
\[
\frac {\de f_1}{\de x}(x, r)\leq 0, \quad \text{sign} \frac {\de f_1}{\de r}(x, r)= -\text{sign} \frac {\de k}{\de x}(x,y) \quad\forall\, x,y,r .
\]
Then the condition \eqref{sign}   on $D_xF=\left(\frac {\de f_1}{\de x} +  \frac {\de f_1}{\de r}
\int_\R \frac {\de k}{\de x}
\mu \,dy  \right) f_2(M(\mu))$ in (FG2) is satisfied and $D_xF(\cdot, \mu)\not\equiv 0$ if in addition either $\frac {\de f_1}{\de x}\ne0$ or both $\frac {\de f_1}{\de r}\ne 0$ and $ \frac {\de k}{\de x}\ne 0$. 
Similar assumptions can be made on $G$.
\end{ex}
About the diffusion coefficient $\s$ we will assume
 \begin{equation}
\label{sigma}
\s : \R \to \R \quad\text{Lipschitz }
 , \qquad \frac 12\s^2(x)\geq \s_o>0 \quad\forall \, x\in \R .
\end{equation}
%
%
\subsection{Non-uniqueness for any time horizon}
In this section we consider the Hamiltonian
 \begin{equation}
\label{Ha}	
H(p) := \max_{a\leq \g \leq b}\{-p\g\} =
\left\{
\begin{array}{ll} 
-bp &\text{ if } p\leq 0, \\ -ap &\text{ if } p\geq 0 ,\end{array}
\right.\,
\end{equation}
so that
\[
H'(p) = -b \text{ if } p<0 , \quad H'(p) = -a \text{ if } p>0 .
\]
 \begin{thm} 
 \label{T1}
 Assume {\rm(F1), 
 (G1), 
 (FG2)}, \eqref{sigma}, and that  $H$ is given by  \eqref{Ha} with $a<0<b$. Then, for all $\nu$ with $M(\nu)=0$, there are two classical solutions $(v_1, m_1), (v_2, m_2)$  of \eqref{pde} such that $(v_1)_x(t,x)<0$ and $(v_2)_x(t,x)>0$ for all $t<T$, $M(m_1(t,\cdot))=bt$ and $M(m_2(t,\cdot))=at$ for all $t$. 
\end{thm}   
 \begin{proof}  We begin with the construction of $(v_1, m_1)$ and drop the subscripts. Observe that if $v_x<0$ the second equation of \eqref{pde} becomes 
 \begin{equation}
\label{m1}
 m_t  + bm_x = \frac 12 (\s^2(x) m)_{xx}  \quad 
\text{ in  }(0, T)\times \R.
\end{equation}
 A solution of this equation with the initial condition $m(0,x)=\nu(x)$ exists by standard results on parabolic equations \cite {Fri}, and it is the law of the process $X(\cdot)$ solving
 \begin{equation*}
	X(t) = \xi + bt + \int_{0}^{t} 
	\sigma(X(s)) d W(s)
	 ,\quad t\in [0,T],
\end{equation*}
where $W$ is a standard one-dimensional Wiener process and $\Law(\xi) = \nu$. 
Then
\begin{equation}
\label{Mm1}
M(m(t)) = \M [X(t)] = M(\nu) + bt >0 \quad \forall t\in (0, T].
\end{equation}
Moreover, there exists $C>0$ such that 
  \begin{equation}
\label{d1}
d_1(m(t), m(s)) \leq C(b+\|\s\|_\infty)\sqrt{|t-s|} ,
\end{equation}
see, e.g., \cite{Car}.
For such $m$ we consider the Cauchy problem 
  \begin{equation}
\label{v1}
 -v_t -bv_x = \frac 12 \s^2(x) v_{xx} + F(x, m(t)) \quad 
 \text{ in  }(0, T)\times \R , \quad  v(T,x)=G(x, m(T)) ,
\end{equation}
which has a unique classical solution by standard results on parabolic equations \cite{Fri}, in view of \eqref{d1} and the assumptions  \rm{(F1)}, \rm{(G1)}, and \eqref{sigma}. If we show that $v_x<0$ this equation coincides with the first PDE in \eqref{pde} and therefore we get the desired solution  $(v_1, m_1)$ of \eqref{pde}. 

We consider $w:=v_x$ and by the assumptions on the data can differentiate the equation \eqref{v1} to get
   \[
 -w_t -bw_x = 
 \s(x) \s(x)_x w_{x} + \frac 12 \s^2(x) w_{xx} + 
 D_xF(x, m(t)) \quad 
 \text{ in  }(0, T)\times \R , \quad 
 w(T,x)=
 D_xG(x, m(T)) .
 \]
By \eqref{Mm1} and (FG2) $D_xF(x, m(t))\leq 0$ and $D_xG(x, m(T))\leq 0$, thus   the comparison principle implies  $w\leq 0$. 

By the Strong Maximum Principle, if $w(t,x)=0$ for some $t<T$ and some $x$, then $w(s,y)=0$ for all $t<s<T$ and all $y$, 
and so  $D_xG(\cdot, m(T))\equiv 0$. Then we get a contradiction  if the  condition $D_xG(x, \mu)\not\equiv 0$ for all $M(\mu)\ne 0$  holds in (MF2), and reach the desired conclusion $w(t,\cdot)<0$ for all $t<T$.

  If, instead, only the condition $D_xF(\cdot, \mu)\not\equiv 0$ for all $M(\mu)\ne 0$  holds in (MF2), 
  from $w(s,y)=0$ for all $t<s<T$ and all $y$ we get a contradiction with the PDE for $w$ in the interval $(t, T)$, because $D_xF(\cdot, m(t))\not\equiv 0$ and all the other terms in the equation are null. This completes the proof of the existence of $(v_1, m_1)$ with the stated properties.

 The  solution $(v_2, m_2)$ is built in a symmetric way. We first solve 
 \begin{equation*}\label{m2}
 m_t  + am_x = \frac 12 (\s^2(x) m)_{xx}, \quad m(0,x)=\nu(x) ,
 \end{equation*}
and use that the solution  is the law of the process  
 \begin{equation*}
\label{EqDynamics2}
X(t) = \xi + at + \int_{0}^{t} 
	\sigma(X(s)) d W(s) ,
	\end{equation*}
to see that $M(m(t)) = M(\nu) + at<0$ for $t>0$. Next, for such $m$ we solve the Cauchy problem
 \begin{equation}
\label{v2}
 -v_t - av_x = \frac 12 \s^2(x) v_{xx} + F(x, m(t)) \quad 
 \text{ in  }(0, T)\times \R , \quad  v(T,x)=G(x, m(T)) .
\end{equation}
We differentiate this equation and use the assumption (FG2) and the Strong Minimum Principle as before to show that $v_x(t,x)>0$ for all $x$ and $t\in (0, T)$. Then for such solution the last equation coincides with the second equation of \eqref{pde}, which completes the proof of the existence of $(v_2, m_2)$ with the stated properties.
%
  \end{proof} 
 \begin {rem} Symmetry. \upshape
Assume $b=-a$, so that $H(p)=b|p|$, and the data are even, i.e.,
\[
F(x,\mu)=F(-x,\mu) ,\quad G(x,\mu)=G(-x,\mu), \quad \s(x)=\s(-x), \qquad \nu(x)=\nu(-x).
\]
Then the solutions $(v_1, m_1)$ and $(v_2, m_2)$ built in the Theorem are even reflection one of the other, i.e.,
\[
v_1(t,x)=v_2(t,-x), \qquad m_1(t,x)=m_2(t,-x),
\]
as it is easy to check in the construction.
\end  {rem} 
 \begin {rem}  \upshape  If $F\equiv 0$ we can drop the sign condition on $a$ and $b$ and consider any initial density $\nu$ such that $\nu$ with $-bT<M(\nu)<-aT$. Then the same proof produces two solutions with $(v_1)_x(t,x)<0$ and $(v_2)_x(t,x)>0$ for all $t<T$, and $M(m_1(T,\cdot))>0$, $M(m_2(T,\cdot))<0$.
\end  {rem} 
%
\subsection{Eventual non-uniqueness with smooth Hamiltonian}

 In this section we consider Hamiltonians that coincide with the one defined by \eqref{Ha} only for $|p|\geq \d>0$ and therefore can be smooth, see the Example \eqref{ex_smooth}. The precise assumption 
  is 
 \begin{equation}
 \label{H2}	
H\in C(\R), \quad  \exists \, \d, b > 0, \, a<0 \, :\,
H(p) = -bp \text{ if } p\leq -\d , \quad H(p) = -ap \text{ if } p\geq \d .
\end{equation}
On the other hand the assumptions on $D_xG$ in (FG2) are strengthened a bit by adding
\begin{description}
\item [(G3)] 
for some $\eps>0$,   $D_xG(\cdot, \mu)\leq -\d$ if $M(\mu)\geq\eps b$, and  $D_xG(\cdot, \mu)\geq \d$ if $M(\mu)\leq\eps a$.
\end{description}
Then we get the existence of two distinct solutions if the time horizon $T$ is larger than $\eps$.
 \begin{thm} 
  \label{T2}
 Assume  {\rm(F1), 
 (G1), 
 (FG2), (G3)}, \eqref{sigma}, and that  $H$  satisfies \eqref{H2}. Then, for all $T\geq\eps$ and $\nu$ with $M(\nu)=0$, there are two classical solutions $(v_1, m_1), (v_2, m_2)$  of \eqref{pde} such that $(v_1)_x(t,x)\leq-\d$ and $(v_2)_x(t,x)\geq\d$ for all $ 0\leq t\leq T 
 $. 
\end{thm}   
 \begin{proof}
 The construction is the same as in the proof of Theorem \ref{T1}. Now we have, by \eqref{Mm1}
 \[
M(m_1(T)) =  bT \geq b\eps , 
\]
so $D_xG(\cdot, \mu)\leq -\d$ by (G3). Then $w:= (v_1)_x$ satisfies $w(T,x)\leq -\d$ for all $x$, and the Maximum Principle implies $(v_1)_x(t,x)=w(t,x)\leq -\d$ for all $ 0\leq t\leq T$. Then, by \eqref{H2}, the equation \eqref{v1} for $v_1$ coincides with the first PDE in \eqref{pde} and the equation \eqref{m1} for $m_1$ coincides with the second PDE in \eqref{pde}. The construction of the second solution is symmetric because $M(m_2(T)) =  aT \leq a\eps$.
 \end{proof}
\begin{ex}
\upshape
\label{ex_smooth}
The Hamiltonian
 \begin{equation*}
H(p) := \max_{
 |\g| \leq 1}\left\{-p\g+\frac12\d(1-\g^2)\right\} =
\left\{
\begin{array}{ll} 
\frac{p^2}{2\d} +\frac{\d}2 , &\text{ if } |p|\leq \d, \\ |p| , &\text{ if } |p|\geq \d ,\end{array}
\right.\,
\end{equation*}
satisfies \eqref{H2} with $-a=b=1$. Note that $H\in C^1(\R)$ and $H'$ is Lipschitz and bounded, so it satisfies the assumptions required for the Hamiltonian in the uniqueness result for short time horizon of Section \ref{short}, see Remark \ref{C11}. A more detailed probabilistic discussion of this example is in Section  \ref{SectSimpleNonzero}.
\end{ex}
\begin{ex}
 \label{ese2}
\upshape
The terminal cost
\[
G(x,\mu)=- \b xM(\mu) +g(\mu) ,\quad \b>0
\]
satisfies (G1). Since $D_xG(x,\mu)=-\b M(\mu)$, the inequality $M(\mu)D_xG(x,\mu)\leq 0$ in (FG2) is satisfied and (G3) holds with the choice $\eps:=\max\{\frac\d {b\beta}, \frac\d {|a|\beta}\}$.


\end{ex}
 %

\section{Probabilistic approach to multiple MFG solutions}
\label{SectProb}

In this section, we give examples of non-uniqueness analogous to those obtained above, but under slightly different regularity assumptions, based on the probabilistic representation of the mean field game. By this we mean that we work directly with the underlying stochastic dynamics of the controlled state process and the corresponding expected costs. A solution of the mean field game is then a couple of control strategy and flow of probability measures satisfying a certain fixed point property, namely: The strategy is optimal for the control problem associated with the flow of probability measures, which in turn coincides with the flow of marginal distributions of the state process under the control strategy. In subsection~\ref{SectProbMFG}, we give two definitions of solution, differing with respect to the admissible strategies (stochastic open-loop vs.\ Markov feedback). The second definition, based on Markov feedback strategies, is more closely related to the \mbox{PDE} characterization \eqref{pde} of the mean field game, see Remark~\ref{RemVerification} below.

Denote by $\mathcal{P}(\mathbb{R})$ the set of all probability measures on the Borel sets of $\mathbb{R}$, and set
\begin{align*}
	&\mathcal{P}_{1}(\mathbb{R})\doteq \left\{\mu\in \mathcal{P}(\mathbb{R}) : \int |x| \mu(dx) < \infty \right\},& & M(\mu)\doteq \int x\, \mu(dx),& & \mu\in \mathcal{P}_{1}(\mathbb{R}). &
\end{align*}
Endow $\mathcal{P}(\mathbb{R})$ with the topology of weak convergence of measures, and $\mathcal{P}_{1}(\mathbb{R})$ with the topology of weak convergence of measures plus convergence of first absolute moments. Notice that
\[
	\tilde{\mathcal{P}}_{1}(\mathbb{R}) \subset \mathcal{P}_{1}(\mathbb{R})
\]
if we identify probability densities with the probability measures they induce. The topology on $\tilde{\mathcal{P}}_{1}(\mathbb{R})$ generated by the Monge-Kantorovich distance coincides with the topology induced by $\mathcal{P}_{1}(\mathbb{R})$.

\subsection{The mean field game}
\label{SectProbMFG}

As above, we consider mean field games in dimension one over a finite time horizon $T > 0$, with drift coefficient of the state dynamics equal to the control action and dispersion coefficient $\sigma\!: \mathbb{R} \rightarrow \mathbb{R}$, assumed to be Lipschitz continuous (hence of sublinear growth), but possibly degenerate. Let $\Gamma \subseteq \mathbb{R}$ be a compact interval, the set of control actions, and let $f\!: \mathbb{R} \times \mathcal{P}(\mathbb{R})\times \Gamma \rightarrow \mathbb{R}$, $g\!: \mathbb{R} \times \mathcal{P}(\mathbb{R}) \rightarrow \mathbb{R}$ be measurable functions with $f(\cdot,\mu,\gamma)$ of polynomial growth uniformly over compacts in $\mathcal{P}(\mathbb{R})\times \Gamma$ and $g(\cdot,\mu)$ of polynomial growth uniformly over compacts in $\mathcal{P}(\mathbb{R})$.

Let $\mathcal{U}$ be the set of triples $((\Omega,\mathcal{F},(\mathcal{F}_{t}),\Prb),u,W)$ such that $(\Omega,\mathcal{F},(\mathcal{F}_{t}),\Prb)$ forms a filtered probability space satisfying the usual hypotheses and carries a $\Gamma$-valued $(\mathcal{F}_{t})$-progressively measurable process $u$ and a one-dimensional $(\mathcal{F}_{t})$-Wiener process $W$. For $\nu\in \mathcal{P}(\mathbb{R})$, let $\mathcal{U}_{\nu}$ denote the set of quadruples $((\Omega,\mathcal{F},(\mathcal{F}_{t}),\Prb),\xi,u,W)$ such that $((\Omega,\mathcal{F},(\mathcal{F}_{t}),\Prb),u,W)\in \mathcal{U}$ and $\xi$ is a real-valued $\mathcal{F}_{0}$-measurable random variable with $\Prb\circ \xi^{-1} = \nu$. 

Let $\nu\in \mathcal{P}(\mathbb{R})$, and let $u \cong ((\Omega,\mathcal{F},(\mathcal{F}_{t}),\Prb),\xi,u,W)\in \mathcal{U}_{\nu}$. Notice that $\xi$ and $W$ are independent since, by definition of $\mathcal{U}_{\nu}$, $\xi$ is $\mathcal{F}_{0}$-measurable, while $W$ is a Wiener process with respect to the filtration $(\mathcal{F}_{t})$. The dynamics of the state process are then given by:
\begin{equation} \label{EqDynamics}
	X(t) = \xi + \int_{0}^{t} u(s)ds + \int_{0}^{t} \sigma\bigl(X(s)\bigr)dW(s),\quad t\in [0,T].
\end{equation}
Since $\sigma$ is Lipschitz, the solution $X = X^{u}$ of Eq.~\eqref{EqDynamics} is 
 uniquely determined up to $\Prb$-indistinguishability. Its law is determined by the law $\Prb\circ(\xi,u,W)^{-1}$.

The costs associated with initial distribution $\nu$, strategy $u\in \mathcal{U}_{\nu}$, initial time $t\in [0,T]$, and a flow of measures $\mathfrak{m}\in \mathcal{M}\doteq \mathbf{C}([0,T],\mathcal{P}(\mathbb{R}))$ are given by
\[
	J(t,\nu,u;\mathfrak{m})\doteq \Mean\left[ \int_{0}^{T-t} f\left(X^{u}(s),\mathfrak{m}(t+s),u(s)\right)ds + g\left(X^{u}(T-t),\mathfrak{m}(T)\right) \right],
\]
where $X^{u}$ is the unique solution of Eq.~\eqref{EqDynamics} under $u$, provided the expected value is finite; otherwise set $J(t,\nu,u;\mathfrak{m})\doteq \infty$. In the above definition of the cost functional, the processes $X^{u}$ and $u$ always start from time zero, while the flow of measures is shifted according to the initial time. In this way, the set $\mathcal{U}_{\nu}$ of admissible control bases does not depend on the initial time, as opposed to the more standard, though equivalent, definition used in, for instance, \cite{flemingsoner}.

If $\nu = \delta_{x}$ for some $x\in \mathbb{R}$, then we can identify $\mathcal{U}_{\delta_x}$ with $\mathcal{U}$. The value function for a flow of measures $\mathfrak{m}\in \mathcal{M}$ is then defined by
\[
	V(t,x;\mathfrak{m})\doteq \inf_{u\in \mathcal{U}} J(t,\delta_{x},u;\mathfrak{m}),\quad (t,x)\in [0,T]\times \mathbb{R}.
\]
Notice that $V(t,x;\mathfrak{m})$ is finite for every $(t,x)\in [0,T]\times \mathbb{R}$ thanks to the growth assumptions on $f$, $g$ and the boundedness of $\Gamma$.

\begin{rem}
\upshape
Let $t\in [0,T]$, $\nu\in \mathcal{P}(\mathbb{R})$, $\mathfrak{m}\in \mathcal{M}$. If $J(t,\nu,u;\mathfrak{m}) < \infty$ for every $u\in \mathcal{U}_{\nu}$, then
\[
	\inf_{\tilde{u}\in \mathcal{U}_{\nu}} J(t,\nu,\tilde{u};\mathfrak{m}) = \int_{\mathbb{R}} V(t,x;\mathfrak{m}) \nu(dx).
\]
\end{rem}

\begin{defn} \label{DefOLSolution}
Let $\nu\in \mathcal{P}(\mathbb{R})$. An \emph{open-loop solution of the mean field game} with initial distribution $\nu$ is a pair $(u,\mathfrak{m})$ such that
\begin{enumerate}[(i)]
	\item $u \cong ((\Omega,\mathcal{F},(\mathcal{F}_{t}),\Prb),\xi,u,W)\in \mathcal{U}_{\nu}$ and $\mathfrak{m}\in \mathcal{M}$;
	
	\item optimality condition: $\infty >  J(0,\nu,u;\mathfrak{m}) = \int_{\mathbb{R}} V(0,x;\mathfrak{m}) \nu(dx)$;
	
	\item mean field condition: $\Prb\circ (X^{u}(t))^{-1} = \mathfrak{m}(t)$ for every $t\in [0,T]$, where $X^{u}$ is the unique solution of Eq.~\eqref{EqDynamics} under $u$.
\end{enumerate}

\end{defn}

We are mainly interested in solutions of the mean field game in Markov feedback strategies. To this end, set
\[
	\mathcal{A} \doteq \left\{ \alpha\!: [0,T]\times \mathbb{R} \rightarrow \Gamma : \alpha \text{ measurable}  \right\}.
\]
For $\alpha\in \mathcal{A}$, $\nu\in \mathcal{P}(\mathbb{R})$, $t_{0}\in [0,T]$ consider the equation:
\begin{equation} \label{EqDynamicsFb}
	X(t) = \xi + \int_{0}^{t} \alpha\bigl(t_{0}+s,X(s)\bigr)ds + \int_{0}^{t} \sigma\bigl(X(s)\bigr)dW(s),\quad t\in [0,T-t_{0}],
\end{equation}
where $W$ is a one-dimensional $(\mathcal{F}_{t})$-Wiener process on some filtered probability space $(\Omega,\mathcal{F},(\mathcal{F}_{t}),\Prb)$ satisfying the usual hypotheses and carrying a real-valued $\mathcal{F}_{0}$-measurable random variable $\xi$ with $\Prb\circ \xi^{-1} = \nu$.

For $\nu\in \mathcal{P}(\mathbb{R})$, $t_{0}\in [0,T]$, let $\mathcal{A}_{\nu,t_{0}}$ denote the set of all $\alpha\in \mathcal{A}$ such that Eq.~\eqref{EqDynamicsFb} with initial distribution $\nu$ and initial time $t_{0}$ possesses a solution that is unique in law.

\begin{rem} 
\label{RemGirsanov}
\upshape
	If $\sigma$ is bounded and such that $\inf_{x\in \mathbb{R}} \sigma(x) > 0$, then, thanks to Girsanov's theorem, $\mathcal{A}_{\nu,t_{0}} = \mathcal{A}$ for all $t_{0}\in [0,T]$, $\nu\in \mathcal{P}(\mathbb{R})$.
\end{rem}

If $\alpha\in \mathcal{A}_{\nu,t_{0}}$, then there exists $u \cong ((\Omega,\mathcal{F},(\mathcal{F}_{t}),\Prb),\xi,u,W) \in \mathcal{U}_{\nu}$ such that
\begin{equation} \label{EqOLFb}
	\alpha\bigl(t_{0}+s,X^{u}(s,\omega)\bigr) = u(s,\omega)
\end{equation}
for  $\mathrm{Leb}_{T-t_{0}}\otimes \Prb$-almost all $(s,\omega)\in [0,T-t_{0}]\times \Omega$, where $\mathrm{Leb}_{T-t_{0}}$ denotes Lebesgue measure on $\mathcal{B}([0,T-t_{0}])$ and $X^{u}$ is the unique solution of Eq.~\eqref{EqDynamics} under $u$. Moreover, by uniqueness in law, if $\tilde{u} \in \mathcal{U}_{\nu}$ is any other stochastic open-loop strategy such that
\begin{equation*}
	\alpha\bigl(t_{0}+s,X^{\tilde{u}}(s,\omega)\bigr) = \tilde{u}(s,\omega) \text{ for  $\mathrm{Leb}_{T-t_{0}}\otimes \tilde{\Prb}$-a.a.\ } (s,\omega)\in [0,T-t_{0}]\times\tilde{\Omega}
\end{equation*}
with $X^{\tilde{u}}$ the unique solution of Eq.~\eqref{EqDynamics} under control $\tilde{u}$ and $\tilde{\Prb}$ the probability measure coming with $\tilde{u}$, then
\[
	\Prb\circ \left(X^{u},u,W\right)^{-1} = \tilde{\Prb}\circ \left(X^{\tilde{u}},\tilde{u},\tilde{W}\right)^{-1}.
\]

We can therefore define the costs associated with initial time $t\in [0,T]$, initial distribution $\nu$, feedback strategy $\alpha\in \mathcal{A}_{\nu,t}$, and a flow of measures $\mathfrak{m}\in \mathcal{M}$ by setting
\[
	J(t,\nu,\alpha;\mathfrak{m})\doteq J(t,\nu,u;\mathfrak{m})
\]
for any stochastic open-loop strategy $u\in \mathcal{U}_{\nu}$ such that Eq.~\eqref{EqOLFb} holds with respect to $\alpha$, $u$ and initial time $t_{0} = t$.

\begin{defn} \label{DefFbSolution}
Let $\nu\in \mathcal{P}(\mathbb{R})$. A \emph{Markov feedback solution of the mean field game} with initial distribution $\nu$ is a pair $(\alpha,\mathfrak{m})$ such that
\begin{enumerate}[(i)]
	\item $\alpha\in \mathcal{A}_{\nu}$ and $\mathfrak{m}\in \mathcal{M}$;
	
	\item optimality condition: $\infty > J(0,\nu,\alpha;\mathfrak{m}) = \int_{\mathbb{R}} V(0,x;\mathfrak{m}) \nu(dx)$;
	
	\item mean field condition: $\Prb\circ (X^{u}(t))^{-1} = \mathfrak{m}(t)$ for every $t\in [0,T]$, where $X^{u}$ is the unique solution of Eq.~\eqref{EqDynamics} under $u$ with $u\in \mathcal{U}_{\nu}$ such that Eq.~\eqref{EqOLFb} holds with respect to $\alpha$ and $u$.
\end{enumerate}

\end{defn}

Two solutions in the sense of Definition~\ref{DefOLSolution} or Definition~\ref{DefFbSolution} are called \emph{equivalent} if their flows of measures coincide. Notice that, by the mean field condition, equivalent solutions have the same initial distribution.

\begin{rem}
 \label{RemVerification}
 \upshape
Let the function $f$ for the combined running costs have the form
\[
	f(x,\mu,\gamma) = l(\gamma) + F(x,\mu)
\]
for some continuous function $l\!:\Gamma \rightarrow \mathbb{R}$ and some function $F$ as in Section~\ref{section pde}. Set
\[
	H(p)\doteq \max_{\gamma\in \Gamma} \{-l(\gamma) - p\gamma\},\quad p\in \mathbb{R}.
\]
Let $\mathfrak{m}\in \mathcal{M}$, and suppose that $v = v_{\mathfrak{m}}$ is a classical solution of the Hamilton-Jacobi-Bellman equation
\begin{equation} \label{EqProbHJB}
	-v_t + H(v_x) = \frac 12 \s^2(x) v_{xx} +F(x, m(t,\cdot)) \text{ in  }[0,T)\times \R
\end{equation}
with terminal condition $v(T,\cdot) = g(\cdot,\mathfrak{m}(T))$. Then $v(t,x) = V(t,x,\mathfrak{m})$ for all $(t,x)\in [0,T]\times \mathbb{R}$. Moreover, if $\alpha\in \mathcal{A}_{\mathfrak{m}(0)}$ is such that the mean field condition of Definition~\ref{DefFbSolution} holds for $\alpha$ and $\mathfrak{m}$ and
\[
	\alpha(t,x)\in \argmax_{\gamma\in \Gamma} \{-l(\gamma) - p\gamma\} \text{ for all } (t,x)\in [0,T)\times \mathbb{R},
\]
then $(\alpha,\mathfrak{m})$ is a solution in the sense of Definition~\ref{DefFbSolution}.
\end{rem}


\subsection{Multiple solutions}
\label{SectProbMultSol}

Here, the space of control actions $\Gamma$ is assumed to be a compact interval; thus, $\Gamma = [a,b]$ for some $a,b\in \mathbb{R}$ with $a < b$. Let $\psi \in \mathbf{C}^{2}(\mathbb{R})$ be such that
\[
	\sup_{x\in \mathbb{R}} e^{-c|x|}\cdot \max\{ |\psi(x)|, |\psi^{\prime}(x)|, |\psi^{\prime\prime}(x)| \} < \infty \text{ for some }c\in (0,\infty),
\]
while for all $x\in \mathbb{R}$,
\begin{subequations} \label{EqPsi}
\begin{align}
	\label{EqPsi+}	b\cdot \psi^{\prime}(x) + \frac{1}{2}\sigma^{2}(x) \psi^{\prime\prime}(x) &> 0, \\
	\label{EqPsi-}	a\cdot \psi^{\prime}(x) + \frac{1}{2}\sigma^{2}(x) \psi^{\prime\prime}(x) &< 0.
\end{align}
\end{subequations}

Here are three examples for choices of $\psi$ such that \eqref{EqPsi} holds:
\begin{exmpl}
	Let $c > 0$, $d\in \mathbb{R}$. Set $\psi(x)\doteq c\cdot x + d$, $x\in \mathbb{R}$. Then \eqref{EqPsi} holds if $a < 0 < b$. 
\end{exmpl}

\begin{exmpl}
	Let $\lambda > 0$. Set $\psi(x)\doteq e^{\lambda x} - 1$, $x\in \mathbb{R}$. Then \eqref{EqPsi+} holds if $b > 0$, and \eqref{EqPsi-} holds if $\sigma$ is bounded and $a < -\frac{\lambda}{2} \sigma^{2}(x)$ for all $x\in \mathbb{R}$. 
\end{exmpl}

\begin{exmpl}
	Set $\psi(x)\doteq \tanh(x)$, $x\in \mathbb{R}$. Then \eqref{EqPsi} holds if $\sigma$ is bounded and $a \leq -\sigma^{2}(x)$, $\sigma^{2}(x) \leq b$ for all $x\in \mathbb{R}$. 
\end{exmpl}

Set
\begin{align*}
	\mathcal{P}_{\exp}(\mathbb{R})&\doteq \left\{\mu\in \mathcal{P}(\mathbb{R}) : \int e^{c x} \mu(dx) < \infty \text{ for all } c\in \mathbb{R} \right\}, \\
	\mathcal{P}_{\psi}(\mathbb{R})&\doteq \left\{\mu\in \mathcal{P}(\mathbb{R}) : \int |\psi(x)| \mu(dx) < \infty \right\}.
\end{align*}
By the growth assumption on $\psi$, $\mathcal{P}_{\exp}(\mathbb{R}) \subset \mathcal{P}_{\psi}(\mathbb{R})$. For $\mu\in \mathcal{P}_{\psi}(\mathbb{R})$, set
\[
	M_{\psi}(\mu)\doteq \int_{\mathbb{R}} \psi(y)\, \mu(dy).
\]
If $\psi(x)\doteq x$, $x \in \mathbb{R}$, then we have $\mathcal{P}_{\psi}(\mathbb{R}) = \mathcal{P}_{1}(\mathbb{R})$ and $M_{\psi}(\mu) = M(\mu)$, the mean value of $\mu\in \mathcal{P}_{1}(\mathbb{R})$, in accordance with the notation introduced in Section \ref{section pde}. 

In this subsection, we assume the running costs $f$ to be of the form
\[
	f(x,\mu,\gamma) = F(x,\mu)
\]
for some $F\!: \mathbb{R}\times \mathcal{P}(\mathbb{R}) \rightarrow \mathbb{R}$ measurable such that for every $\mu\in \mathcal{P}_{\psi}(\mathbb{R})$,
\[
	F(\cdot,\mu) \text{ is }
	\begin{cases}
	\text{ decreasing if } M_{\psi}(\mu) > 0,\\
	\text{ increasing if } M_{\psi}(\mu) < 0,
	\end{cases}
\]
and the terminal costs $g$ to be given by
\[
	g(x,\mu) = G(x,\mu)
\]
for some $G\!: \mathbb{R}\times \mathcal{P}(\mathbb{R}) \rightarrow \mathbb{R}$ measurable such that for every $\mu\in \mathcal{P}_{\psi}(\mathbb{R})$,
\[
	G(\cdot,\mu) \text{ is }
	\begin{cases}
	\text{ strictly decreasing if } M_{\psi}(\mu) > 0,\\
	\text{ strictly increasing if } M_{\psi}(\mu) < 0.
	\end{cases}
\]
We set $F(\cdot,\mu)\doteq 0$, $G(\cdot,\mu)\doteq 0$ if $\mu\in \mathcal{P}(\mathbb{R}) \setminus \mathcal{P}_{\psi}(\mathbb{R})$. Moreover, $F(\cdot,\mu)$, $G(\cdot,\mu)$ are assumed to be of polynomial growth uniformly over compacts in $\mathcal{P}(\mathbb{R})$. Note that these monotonicity conditions on $F$ and $G$ are very similar to the assumption (FG2) of Section \ref{section pde}.

\begin{prop} \label{PropMultSol}
Grant the hypotheses above, and assume that either $\sigma$ is bounded or that $\psi$ together with its first two derivatives is of polynomial growth. Let $\nu\in \mathcal{P}_{\exp}(\mathbb{R})$ be such that $M_{\psi}(\nu) = 0$. Then there exist two non-equivalent open-loop solutions of the mean field game with initial distribution $\nu$, and the associated value functions are different and not constant.
\end{prop}

\begin{proof}
Let $((\Omega,\mathcal{F},(\mathcal{F}_{t}),\Prb),\xi,W)$ be such that $(\Omega,\mathcal{F},(\mathcal{F}_{t}),\Prb)$ forms a filtered probability space satisfying the usual hypotheses with $W$ a one-dimensional $(\mathcal{F}_{t})$-Wiener process and $\xi$ a real-valued $\mathcal{F}_{0}$-measurable random variable such that $\Prb\circ \xi^{-1} = \nu$. Let $u_{+}\equiv b$ be the constant process equal to $b$, and let $u_{-}\equiv a$ be the constant process equal to $a$. By construction, $((\Omega,\mathcal{F},(\mathcal{F}_{t}),\Prb),\xi,u_{\pm},W) \in \mathcal{U}_{\nu}$. Let $X_{+}$ be the unique solution of Eq.~\eqref{EqDynamics} with $u = u_{+}$, and let $X_{-}$ be the unique solution of Eq.~\eqref{EqDynamics} with $u = u_{-}$. Thus, for all $t\in [0,T]$,
\begin{align*}
	X_{+}(t) 
	= \xi + t\cdot b + \int_{0}^{t} \sigma\bigl(X_{+}(s)\bigr)dW(s), \quad 
	X_{-}(t) 
	= \xi + t\cdot a + \int_{0}^{t} \sigma\bigl(X_{-}(s)\bigr)dW(s).
\end{align*}
By the Burkholder-Davis-Gundy inequalities, the sublinear growth of $\sigma$, the exponential integrability of $\xi$ and Gronwall's lemma,
\begin{equation} \label{EqMaxPolMoments}
	\Mean\left[ \sup_{t\in [0,T]} \left| X_{\pm}(t) \right|^{p} \right] < \infty \text{ for every } p \geq 1.
\end{equation}

If $\sigma$ is bounded, then it also holds that
\begin{equation}
 \label{EqMaxExpMoments}
	\Mean\left[ \sup_{t\in [0,T]} \exp\left( c\cdot X_{\pm}(t) \right) \right] < \infty \text{ for every } c \in \mathbb{R}.
\end{equation}
To check \eqref{EqMaxExpMoments},  observe that, thanks to the monotonicity and positivity of the exponential and the Cauchy-Schwarz inequality,
\[
	\Mean\left[ \sup_{t\in [0,T]} \exp\left( c\cdot X_{\pm}(t) \right) \right] \leq e^{|c|\max\{|a|,|b|\}T} \cdot \sqrt{\Mean\left[ e^{2c\xi}\right]} \cdot \sqrt{ \Mean\left[ \sup_{t\in [0,T]} \exp\left( 2c \int_{0}^{t} \sigma\bigl(X_{\pm}(s)\bigr)dW(s) \right) \right] }.
\]
The first expected value on the right-hand side above is finite by hypothesis. For the second, we have, by the boundedness of $\sigma$,
\begin{align*}
	&\Mean\left[ \sup_{t\in [0,T]} \exp\left( 2c \int_{0}^{t} \sigma\bigl(X_{\pm}(s)\bigr)dW(s) \right) \right] \\
	&\leq \Mean\left[ \sup_{t\in [0,T]} \exp\left( c \int_{0}^{t} \sigma\bigl(X_{\pm}(s)\bigr)dW(s) - \frac{c^{2}}{2} \int_{0}^{t} \sigma^{2}\bigl(X_{\pm}(s)\bigr) ds \right)^{2} \right] \cdot e^{c^{2} \|\sigma\|_{\infty}^{2} T} \\
	&\leq 4\Mean\left[ \exp\left( 2c \int_{0}^{T} \sigma\bigl(X_{\pm}(s)\bigr)dW(s) - c^{2} \int_{0}^{T} \sigma^{2}\bigl(X_{\pm}(s)\bigr) ds \right) \right] \cdot e^{c^{2} \|\sigma\|_{\infty}^{2} T} \\
	&\leq 4 \underbrace{\Mean\left[ \exp\left( 2c \int_{0}^{T} \sigma\bigl(X_{\pm}(s)\bigr)dW(s) - \frac{4c^{2}}{2} \int_{0}^{T} \sigma^{2}\bigl(X_{\pm}(s)\bigr) ds \right) \right]}_{=1} \cdot e^{2c^{2} \|\sigma\|_{\infty}^{2} T} < \infty,
\end{align*}
where we have used Doob's maximal inequality (second to third line) as well as the fact that the stochastic exponential of a martingale with bounded quadratic variation process is again a martingale; see, for instance, Section~3.5D in \cite{karatzasshreve91}.

Set
\begin{align*}
	& \mathfrak{m}_{+}(t)\doteq \Prb\circ \left(X_{+}(t)\right)^{-1}, & & \mathfrak{m}_{-}(t)\doteq \Prb\circ \left(X_{-}(t)\right)^{-1}, & & t\in [0,T]. &
\end{align*}
Clearly, $\mathfrak{m}_{+}(0) = \nu = \mathfrak{m}_{-}(0)$. By Eq.~\eqref{EqMaxPolMoments} and the dominated convergence theorem, the mappings $[0,T]\ni t\mapsto \mathfrak{m}_{\pm}(t)\in \mathcal{P}(\mathbb{R})$ are continuous. It follows that $\mathfrak{m}_{\pm} \in \mathcal{M}$ and that $M_{\psi}(\mathfrak{m}_{\pm}(t))$ is finite for every $t\in [0,T]$.

We are going to show that
\begin{equation} \label{EqProofOptCond}
	J(0,\nu,u_{+};\mathfrak{m}_{+}) = \inf_{\tilde{u}\in \mathcal{U}_{\nu}} J(0,\nu,\tilde{u};\mathfrak{m}_{+}).
\end{equation}
This will imply the optimality condition of Definition~\ref{DefOLSolution}. Since the mean field condition is satisfied by construction of $\mathfrak{m}_{+}$, it will follow that $(u_{+},\mathfrak{m}_{+})$ is an open-loop solution of the mean field game with initial distribution $\nu$.

By assumption and construction, $M_{\psi}(\nu) = 0 = M_{\psi}\left(\mathfrak{m}_{+}(0)\right)$. By It{\^o}'s formula, the Fubini-Tonelli theorem, the growth conditions on $\sigma$ and $\psi$ and its derivatives, and by Eq.~\eqref{EqMaxPolMoments} and Eq.~\eqref{EqMaxExpMoments}, respectively, we have for all $t\in [0,T]$,
\[
	M_{\psi}\left(\mathfrak{m}_{+}(t)\right) =  0 + \int_{0}^{t} \Mean\left[ b\cdot \psi^{\prime}(X_{+}(s)) + \frac{1}{2}\sigma^{2}(X_{+}(s)) \psi^{\prime\prime}(X_{+}(s)) \right] ds.
\]
By \eqref{EqPsi+}, it follows that $M_{\psi}\left(\mathfrak{m}_{+}(t)\right) > 0$ whenever $t > 0$. As a consequence, the functions $F(\cdot,\mathfrak{m}_{+}(t))$, $t\in (0,T)$, $G(\cdot,\mathfrak{m}_{+}(T))$ are decreasing (with strict decrease for the terminal costs).

Let $\tilde{u}\in \mathcal{U}_{\nu}$. We may assume, since $X_{+}$ is measurable with respect to the filtration generated by the initial condition and the Wiener process, that the strategies $\tilde{u}$, $u_{+}$ are defined on the same stochastic basis with the same driving Wiener process and the same initial condition. Thus, with a slight abuse of notation, $\tilde{u} \cong ((\Omega,\mathcal{F},(\mathcal{F}_{t}),\Prb),\xi,\tilde{u},W)$, $u_{+} \cong ((\Omega,\mathcal{F},(\mathcal{F}_{t}),\Prb),\xi,u_{+},W)$. Let $\tilde{X}$ be the unique solution of Eq.~\eqref{EqDynamics} with $u = \tilde{u}$, and let $X_{+}$ be the unique (strong) solution of Eq.~\eqref{EqDynamics} with $u = u_{+}$, as before. Since $b \geq \tilde{u}(t,\omega)$ for all $(t,\omega)\in [0,T]\times \Omega$, Theorem~1.1 in \cite{ikedawatanabe77} entails that
\begin{equation} \label{EqMonotoneXpm}
	\Prb\left( X_{+}(t) \geq \tilde{X}(t) \text{ for all } t\in [0,T] \right) = 1.
\end{equation}
This implies, in view of the monotonicity of the costs, that
\[
	J(0,\nu,u_{+};\mathfrak{m}_{+}) \leq J(0,\nu,\tilde{u};\mathfrak{m}_{+}),
\]
which yields the optimality condition \eqref{EqProofOptCond} for $(u_{+},\mathfrak{m}_{+})$. The optimality condition for $(u_{-},\mathfrak{m}_{-})$ is established in a completely analogous way by using \eqref{EqPsi-} instead of \eqref{EqPsi+} and the opposite monotonicity of the costs.

We have shown that $(u_{+},\mathfrak{m}_{+})$, $(u_{-},\mathfrak{m}_{-})$ are two open-loop solutions of the mean field game with the same initial distribution. These two solutions are non-equivalent since $M_{\psi}\left(\mathfrak{m}_{+}(t)\right) > 0$ while $M_{\psi}\left(\mathfrak{m}_{-}(t)\right) < 0$ whenever $t > 0$. Moreover, the associated value functions $V(\cdot,\cdot,\mathfrak{m}_{+})$, $V(\cdot,\cdot,\mathfrak{m}_{-})$ are different and non-constant since $G(\cdot,\mathfrak{m}_{+}(T))$ is strictly decreasing while $G(\cdot,\mathfrak{m}_{-}(T))$ is strictly increasing.
\end{proof}

The proof of Proposition~\ref{PropMultSol} also shows that, under the hypotheses of the proposition, there exist two non-equivalent Markov feedback solutions of the mean field game with non-constant value function.

In Section~\ref{SectSimpleNonzero}, we will consider a variant with non-zero running costs of the following simple example with zero running costs:

\begin{exmpl} \label{ExmplThreeSol}
\upshape
 Choose $\Gamma \doteq [-1,1]$ for the space of control actions, $f\equiv F \equiv 0$ as the running costs, and $\sigma(\cdot) \equiv \sigma$ for some constant $\sigma \in [0,\infty)$ as dispersion coefficient. Set $\psi(x)\doteq x$, $x \in \mathbb{R}$; thus $M_{\psi}(\mu) = M(\mu)$ is the mean of $\mu\in \mathcal{P}_{\psi}(\mathbb{R}) = \mathcal{P}_{1}(\mathbb{R})$. Set $G(x,\mu)\doteq -M(\mu)\cdot x$, $x \in \mathbb{R}$, if $\mu\in \mathcal{P}_{1}(\mathbb{R})$. Let $\nu\in \mathcal{P}_{\exp}(\mathbb{R})$ be such that $M_{\psi}(\nu) = 0$. By Proposition~~\ref{PropMultSol}, there exist two non-equivalent solutions of the mean field game with non-constant value function. In addition, there exists a third solution, namely the trivial solution (open-loop or feedback) corresponding to the constant strategy equal to zero; the value function in this case is constant and equal to zero as well. Observe that these three non-equivalent solutions exist for any time horizon $T > 0$, be it small or large.
\end{exmpl}


\subsection{Simple example with regular Hamiltonian}
 \label{SectSimpleNonzero}

In this section, we work with the following data:
\begin{itemize}
	\item set of control actions $\Gamma = [-1,1]$;
	
	\item running costs $f$ given by $f(x,\mu,\gamma) \doteq c_{0}\cdot |\gamma|^{2}$ for some constant $c_{0} > 0$;
	
	\item terminal costs $g(x,\mu) \doteq -M(\mu)\cdot x$;
	
	\item dispersion coefficient $\sigma(\cdot) \equiv \sigma$ for some constant $\sigma \in [0,\infty)$.
\end{itemize}
The associated Hamiltonian $H$ is therefore
\[
	H(p) \doteq \max_{\gamma\in [-1,1]} \left\{-c_{0}\gamma^2 - p\gamma \right\} = -c_{0}\min\left\{ \frac{p^2}{4c_0^2},1\right\} + |p| \cdot \min\left\{ \frac{|p|}{2c_0},1\right\},\quad p\in \mathbb{R}.
\]
Note that it differs from Example \ref{ex_smooth} only by a constant.

For this simple example, the value function associated with a given flow of measures can be computed explicitly. For $M\in \mathbb{R}$, let $\alpha_{M} \in \mathcal{A}$ denote the constant Markov feedback strategy given by
\[
	\alpha_{M} \equiv 
	\sgn(M)\cdot \min\left\{ \frac{|M|}{2c_0},1\right\}.
\]
In particular, $\alpha_{0}\equiv 0$.

Let $\mathfrak{m}\in \mathcal{M}$. Set $M\doteq M(\mathfrak{m}(T))$. Then, for all $(t,x)\in [0,T]\times \mathbb{R}$,
\begin{align*}
	J(t,\delta_{x},\alpha_{M};\mathfrak{m}) &{ = c_{0}\int_{0}^{T-t} \left(\min\left\{ \frac{|M|}{2c_0},1\right\}\right)^{2} dt - M\cdot \left(x + \int_{0}^{T-t} \sgn(M) \left(\min\left\{ \frac{|M|}{2c_0},1\right\}\right) dt \right) } \\
	&= -M\cdot x + (T-t)\left(c_{0}\min\left\{ \frac{M^2}{4c_0^2},1\right\} - |M| \cdot \min\left\{ \frac{|M|}{2c_0},1\right\} \right).
\end{align*}
The function $(t,x)\mapsto J(t,\delta_{x},\alpha_{M};\mathfrak{m})$ satisfies the Hamilton-Jacobi-Bellman equation \eqref{EqProbHJB} with terminal condition $g(\cdot,\mathfrak{m}(T))$. It follows that $V(t,x,\mathfrak{m}) = J(t,\delta_{x},\alpha_{M};\mathfrak{m})$; cf.\ Remark~\ref{RemVerification}. In particular, $V(\cdot,\cdot,\mathfrak{m}) \equiv 0$ if $M = 0$.

Let $\nu \in \mathcal{P}_{1}(\mathbb{R})$. Let $X^{\alpha_{M}}$ be a solution of Eq.~\eqref{EqDynamicsFb} with feedback strategy $\alpha = \alpha_{M}$, initial distribution $\nu$ and initial time $t_{0} = 0$. Then
\[
	\Mean\left[ X^{\alpha_{M}}(T) \right] = M(\nu) + \sgn(M)\cdot T\cdot \min\left\{ \frac{|M|}{2c_0},1\right\}.
\]
Observe that $X^{\alpha_{M}}$ is the unique in law optimal process for the flow of measures $\mathfrak{m}$, where $M = M(\mathfrak{m}(T))$. Define the flow of measures $\mathfrak{m}_{\nu,M}$ by
\[
	\mathfrak{m}_{\nu,M}(t)\doteq \Law\left(X^{\alpha_{M}}(t)\right),\quad t\in [0,T].
\]
Notice that $\mathfrak{m}_{\nu,M}$ depends only on $\nu$ and the value of $M$. In view of the mean field condition in Definition~\ref{DefFbSolution} (or Definition~\ref{DefOLSolution}), it follows that $\mathfrak{m}$ is the flow of measures of a solution of the mean field game with initial distribution $\nu$ if and only if $\mathfrak{m} = \mathfrak{m}_{\nu,M}$ and
\begin{equation} \label{EqSimpleConsistency}
	M(\nu) = M - \sgn(M)\cdot T\cdot \min\left\{ \frac{|M|}{2c_0},1 \right\}.
\end{equation}
In this case, the solution is given by $(\alpha_{M},\mathfrak{m}) = (\alpha_{M},\mathfrak{m}_{\nu,M})$. The number of solutions of the mean field game with initial distribution $\nu$ is therefore equal to the number of solutions of Eq.~\eqref{EqSimpleConsistency} as $M$ varies over $\mathbb{R}$. We distinguish the following cases:
\begin{enumerate}[a)]
	\item Eq.~\eqref{EqSimpleConsistency} has at most one solution with $M = 0$, namely one if $M(\nu) = 0$, else zero.
	
	\item $|M| \in (0,2c_{0})$: In this domain, Eq.~\eqref{EqSimpleConsistency} has at most one solution if $T \neq 2c_{0}$, namely one if $M(\nu) = (1-\frac{T}{2c_{0}})M$, else zero. If $T = 2c_{0}$, then Eq.~\eqref{EqSimpleConsistency} has infinitely many solutions if $M(\nu) = 0$, else zero.
	
	\item $|M| \geq 2c_{0}$: In this domain, Eq.~\eqref{EqSimpleConsistency} has at most one solution, namely one if $M(\nu) = M - \sgn(M)T$, else zero.
	
	
\end{enumerate}

The number of solutions of the mean field game thus depends on the time horizon $T$; the critical value is $T = 2c_{0}$:

\paragraph{Small time horizon:} $T < 2c_{0}$. In this case, the mean field game possesses exactly one solution with initial distribution $\nu$. The unique solution is given by $(\alpha_{M},\mathfrak{m}_{\nu,M})$ with
\[
	M = \begin{cases}
		M(\nu)/\left(1 - \frac{T}{2c_{0}} \right) &\text{if } |M(\nu)| \in [0,2c_{0}-T), \\
		M(\nu) + \sgn(M(\nu))T &\text{if } |M(\nu)| \geq 2c_{0}-T.
	\end{cases}
\]

\paragraph{Critical time horizon:} $T = 2c_{0}$. In this case, the mean field game possesses either one or infinitely many solutions, depending on the initial distribution $\nu$: There are infinitely many solutions if $M(\nu) = 0$, namely the solutions corresponding to any $M \in [-2c_{0},2c_{0}]$; there is exactly one solution if $M(\nu) \neq 0$, namely the solution corresponding to $M = M(\nu) + \sgn(M(\nu))T$.

\paragraph{Large time horizon:} $T > 2c_{0}$. In this case, the mean field game possesses one, two, or three non-equivalent solutions, depending on the initial distribution $\nu$: There are exactly three solutions if $|M(\nu)| < T - 2c_{0}$, namely the solutions corresponding to $M\in \{M(\nu)/(1 - \frac{T}{2c_{0}}), M(\nu)+T, M(\nu)-T\}$. There are exactly two solutions if $|M(\nu)| = T - 2c_{0}$, namely the solutions corresponding to $M\in \{M(\nu)+T, M(\nu)-T\}$. There is exactly one solution if $|M(\nu)| > T - 2c_{0}$, namely the solution corresponding to $M = M(\nu) + \sgn(M(\nu))T$.


\section{Comparing the examples of non-uniqueness with some uniqueness results}
\label{uniqueness}
In this section we assume for simplicity that $\s$ is constant, and then we can consider without loss of generality 
$$ \frac 12 \s^2(x)=1 .$$
First we state a uniqueness result under the classical monotonicity condition on the costs of Lasry and Lions \cite{LL} and see that such condition is essentially sharp in view of the results of Sections \ref{section pde} and \ref{SectProb}. In the following section we prove a 
uniqueness theorem under general assumptions 
 on the costs and merely smoothness of the Hamiltonian, provided the time horizon $T$ is short enough, and compare it with our examples of multiple solutions for all horizons $T>0$.

\subsection{Uniqueness under monotonicity conditions}
\label{monotonicity}

In the next theorem we will assume, for all measures $\mu, \nu\in \mathcal{P}_{1}(\mathbb{R})$ admitting a $C^2$ density
  \begin{equation}\label{monoG}
  \int_\R (G(x,\mu) - G(x,\nu) ) d(\mu -\nu)(x) \geq 0 , 
  \end{equation}
    \begin{equation}\label{monoF}
  \int_\R (F(x,\mu) - F(x,\nu)) d(\mu -\nu)(x) ) > 0 , \quad \text{if } \; M(\mu)\ne M(\nu) ,
  \end{equation}
    \begin{equation}
    \label{FG3}
  \text{for any fixed } \mu , \; x\mapsto F(x,\mu) \;\text{ and } x\mapsto G(x,\mu) \; \text{ depend only on }\; M(\mu) ,
    \end{equation}
    \begin{equation}
    \label{FG4}
  |F(x,\mu)| + |G(x,\mu)| \leq C_R(1+|x|) , \quad \text{if } \; \int_\R |x| d\mu(x) \leq R .
      \end{equation}
      We call the property \eqref{monoG} {\em weak monotonicity} of $G$.
     \begin{thm}
     \label{th:uniq}
Assume \eqref{monoG}, \eqref{monoF}, \eqref{FG3}, \eqref{FG4}, $\s$ constant, $H$ 
 convex and Lipschitz, and $\nu \,dx\in 
{ \mathcal{P}}_1(\R)$.  Let $(v_1, m_1), (v_2, m_2)$ be two classical solutions of \eqref{pde} such that $v:=v_1-v_2$ has  
$v_t, v_x$, and $v_{xx}$ bounded by $C(1+|x|)$. Then  $v_1
 = v_2
 $ and $m_1
  = m_2
  $  
  in $[0,T]\times \R$. 
 \end{thm}
  \begin{proof} We only make some little variants to the proof of Theorem 3.6 of \cite{Car}. Using the representation of $m_i$ as the law of a diffusion process with drift bounded by $\|H'\|_\infty$, as in the proof of Theorem \ref{T1}, by standard properties of the Ito integral we get
  \[
  \int_\R |x| dm_i(t,x) \leq \int_\R |x| d\nu(x) + T\|H'\|_\infty + |\s|\sqrt T .
  \]
       Then $m_i(t, x) \,dx\in { \mathcal{P}}_1(\R)$ for all $t$.
       
  We set 
  $m:=m_1-m_2$.  We write the equation 
   for $m$
  \[
  m_t - m_{xx} -\text{div}(m_1H'((v_1)_x) - m_2H'((v_2)_x) = 0 ,
  \]
 multiply it by  $v \zeta_r$, where $\zeta_r$ is a cutoff function in 
 space with $\zeta_r(x)=1$ for $|x|\leq r$, and integrate by parts. Then we take the equation  
  for $v$
  \[
   -v_t+H((v_1)_x) -H((v_2)_x)= v_{xx}+F(x,m_1)-F(x,m_2) ,
  \]
  multiply it by $m$, integrate in time and space 
  and add it to the previous one. Next we let $r\to\infty$, which is possible by \eqref{FG4} and the growth conditions on the derivatives of $v$. As in \cite{Car}, by \eqref{monoG} and the convexity of $H$ we get
  \[
 \int_0^T \int_\R (F(x,m_1(t)) - F(x,m_2(t)) d(m_1-m_2)(x) \, dt \leq 0 .
  \]
Then   \eqref{monoF} implies $M(m_1(t))=M(m_2(t))$ for all $t$. Now by \eqref{FG3} and the uniqueness results for HJB equations we obtain $v_1=v_2$. We plug this into the KFP equation and finally get $m_1=m_2$.
   \end{proof}
 \begin{rem}
 \upshape	
 Some variants of this theorem and more details on the proof can be found in \cite{Far}. 
 \end{rem}
\begin{cor} 
\label{comparison} 
 Assume  the Hamiltonian  is given by \eqref{Ha}, $\s$ is constant, and the costs are
       \begin{equation}
    \label{costs_ex}
   F(x,\mu)= \a x M(\mu) + f(\mu) , \quad   G(x,\mu)= \beta xM(\mu) + g(\mu) ,
         \end{equation}  
with $f, g : \tilde{ \mathcal{P}}_1(\R) \to \R$  continuous for $d_1$.   Then 
             \begin{enumerate}
         \item [i)] there is at most one classical solution of \eqref{pde}  with 
         derivatives bounded by $C(1+|x|)$ if $\nu \,dx\in { \mathcal{P}}_1(\R)$
          and 
 \[        \a > 0 , \quad \beta \geq 0 \,;
 \]
    \item [ii)] there are two distinct classical solutions of \eqref{pde} with 
    derivatives bounded by $C(1+|x|)$ if $\nu \in \tilde{ \mathcal{P}}_1(\R)$, $M(\nu)=0$, 
    and 
 \[        \a \leq 0 , \quad \beta < 0. \]
         \end{enumerate}
     \end{cor} 
 \begin{proof}   {\em i)}  First observe that
 \[
 \int_\R (f(\mu) - f(\nu))
 (\mu -\nu)(x) \, dx= 0
 \]
 because $ \int_\R  \mu(x) \, dx= \int_\R  \nu(x) \, dx =1$.
 Then for the costs of the form \eqref{costs_ex} we 
 compute 
  \[
  \int_\R (F(x,\mu) - F(x,\nu))d(\mu -\nu)(x) = \a(M(\mu) -M(\nu))   \int_\R x d(\mu -\nu)
   = \a (M(\mu) -M(\nu)) ^2 ,
  \]
  so  \eqref{monoF} is satisfied if $\a>0$. Similarly, $G$ satisfies \eqref{monoG} if $\beta\geq 0$. Then we get the conclusion by Theorem \ref{th:uniq}.

    {\em ii)} This comes from Theorem \ref{T1} in view of Example \ref{ese1}. The growth condition on the solutions $v_i$ follows from the equations \eqref{v1} and \eqref{v2}, by \eqref{FG4} and the boundedness of $M(m_i(t))$.
\end{proof}    
\begin{rem}
 \upshape	

 The last result can be modified to cover possibly smooth Hamiltonians satisfying merely \eqref{H2} instead of  \eqref{Ha}. In fact, statement {\em i)} remains true if $H$ is convex, whereas {\em ii)} holds for $T>\max\{\frac\d {b\beta}, \frac\d {|a|\beta}\}$ by Theorem \ref{T2} and Example \ref{ese2}.
    \end{rem}
\begin{rem}
\label{crowd}  
 \upshape	
The assumptions in the two cases of the Corollary have the following simple interpretation in terms of aversion or not to the crowd. The term $x M(\mu)$ is positive if the position $x$ of a representative individual in the population is on the same side of the origin as the expected value of the population distribution. Then $\a > 0$ and $\beta>0$ mean that the individual pays a cost for such a position, whereas he has a gain if he stays on the opposite side. This models a form of aversion to crowd, whereas the case $\a < 0$ and $\beta<0$  corresponds to a reward for having positions on the same side as the mean position of the population. The fact that the monotonicity conditions for uniqueness are related to crowd-aversion is well known \cite{GLL}, and there are examples of multiple solutions when imitation is rewarding in stationary MFG \cite{Gu, B}. The present example seems to be the first for evolutive MFG PDEs, together with those in the very recent papers \cite{BriCar} and \cite {CT}.
\end{rem}

\begin{rem} \upshape Note  that in the local case with costs of the form 
 $F(x,\mu)= f(x) \mu(x)$ 
 the crucial quantity for uniqueness is the sign of $f$ ($f>0$ implies uniqueness), whereas in the non-local case
  $F(x,\mu)= f(x)M(\mu)$ 
   what seems to count is the sign of $f'$, because  $f'<0$ implies non-uniqueness.
   \end{rem}
  %


\subsection{Uniqueness for short time horizon} 
\label{short}
In this section we assume 
$ \frac 12 \s^2(x)=1$. On the other hand we allow any space dimension $d\geq 1$. We consider the MFG system 
\begin{equation}\label{mfg-sys}
\left\{
\begin{array}{lll}
-v_t + H(Dv) = \Delta v
+ F(x, m(t,\cdot))  \quad 
 \text{ in  }(0, T)\times \R^d, \quad & v(T,x)=G(x, m(T))) ,\\ \\
m_t - \text{div}(
DH(Dv) m) = \Delta m
  \quad 
\text{ in  }(0, T)\times \R^d, \quad\ & m(0,x)=\nu(x) ,
 \end{array}
\right.\,
\end{equation}
where $Dv
=\nabla_x v
$ denotes the gradient of $v$ with respect to the space variables, $\Delta$ is the Laplacian with respect to the space variables $x$, and $DH=\nabla _pH$ is the gradient of the Hamiltonian. Our main assumptions are the smoothness of the Hamiltonian and a 
Lipschitz continuity of the costs in the norm $\|\cdot\|_2$ of $L^2(\R^d)$ that we state next. Define
\[
\tilde{ \mathcal{P}}(\R^d) := \{\mu\in L^\infty(\R^d) : \mu\geq 0,\; \int_\R \mu(x) dx = 1 
\},
\]
and note that $\tilde{ \mathcal{P}}(\R^d)\subseteq L^2(\R^d)$.
We will assume $F, G : \tilde{\mathcal{P}}
(\mathbb{R}^d)
 \to \R$ satisfy, for all $\mu, \nu$, 
\begin{equation}\label{Flip}
\|F(\cdot,\mu)-F(\cdot,\nu)\|_2^2 \leq L_F \|\mu-\nu\|_2^2 , 
\end{equation}
\begin{equation}\label{Glip}
\|DG(\cdot,\mu)-DG(\cdot,\nu)\|_2^2 \leq L_{
G} \|\mu-\nu\|_2^2 , 
\end{equation}
The next theorem 
is a variant and an extension of a result presented by Lions in \cite{L}, see Remark \ref{plions} for a comparison.
\begin{thm}
\label{st-uniq} 
Assume $H\in C^2(\R^d)$, \eqref{Flip}, \eqref{Glip}, $\nu\in \tilde{ \mathcal{P}}
(\R^d)$, and $(v_1, m_1), (v_2, m_2)$ are two classical solutions of \eqref{mfg-sys} 
such that $D(v_1-v_2)\in L^2([0,T]\times \R^d)$.
Suppose also that either

\noindent (i)  $|DH| \leq C_H$ and $ |D^2 H|\leq \bar C_H$, or

\noindent (ii) $|Dv_i|\leq K$, $i= 1, 2$.

\noindent
Then there exists $\bar T>0$ such that, if $T<\bar T$, $v_1(t, \cdot) = v_2(t, \cdot)$ and $m_1(t, \cdot) = m_2(t, \cdot)$ for all $t\in[0,T]$. Moreover $\bar T$ depends only on $d, L_F, L_G, \|\nu\|_\infty,$  $C_H,$ and $\bar C_H$ in case (i), and only on $d, L_F, L_G, \|\nu\|_\infty$, $\sup_{|p|\leq K}|DH(p)|$, and $\sup_{|p|\leq K}|D^2H(p)|$ in case (ii).
 \end{thm}
 \begin{proof}
  \noindent {\em Step 0.}  We first use the density estimate of Proposition \ref{prop_app} in the Appendix to get that $m_i(t,\cdot)\in  \tilde{ \mathcal{P}}(\R^d)$ for all $t$ and
 \begin{equation}
 \label{est_m_bis}
  \|m_i(t,\cdot)\|_\infty\leq C\, \|\nu\|_\infty ,  \quad i=1, 2, \;\forall \,t\in[0,T],
 \end{equation}
 where the constant $C = C(T,\|DH(Dv_i)\|_\infty, d)$ is bounded for bounded entries. By known properties of the Fokker-Planck equation for $m_i$, we have that $m_i(t,\cdot)\in \tilde{ \mathcal{P}}$ for all $t$ and then
 \[
 \sup_{0\leq t\leq T}  \|m_i(t,\cdot)\|_2\leq  \sup_{0\leq t\leq T}  \|m_i(t,\cdot)\|_\infty <+\infty .
 \]
 %
 
\noindent
  {\em Step 1.} Define $v:=v_1-v_2$, $m:=m_1-m_2$, $B(t,x):=\int_0^1DH(Dv_2+s(Dv_1-Dv_2)) ds$, and observe that $B\in L^\infty((0,T)\times\R^d)$ and $v$ satisfies
 \begin{equation*}
\left\{
\begin{array}{lll} -v_t+B(t,x)\cdot Dv=\Delta v+F(x,m_1)-F(x,m_2)
\quad 
 \text{ in  }(0, T)\times \R^d, 
 \\
 v(T,x)=G(x, m_1(T)) - G(x, m_2(T)) .
  \end{array}
\right.\,
 \end{equation*}
Now set $w:=Dv$, $\bar F(t,x):=F(x,m_1)-F(x,m_2)$, and differentiate the equation to get the system of parabolic PDEs
 \begin{equation}\label{syst}
\left\{
\begin{array}{lll} -\frac{\de w_j}{\de t}+(B(t,x)\cdot w)_{x_j}=\Delta w_j+(\bar F(t,x))_{x_j}
\quad 
 \text{ in  }(0, T)\times \R^d, \quad j=1,\dots, d , 
 \\ \\
 w(T,x)=D_xG(x, m_1(T)) - D_xG(x, m_2(T)) .
  \end{array}
\right.\,
 \end{equation}

\noindent {\em Step 2.} Take a smooth cutoff function $\zeta_R\geq 0$ satisfying $|D\zeta_R|\leq C$, $\zeta_R(x)=1$ for $|x|\leq R$ and $\zeta_R(x)=0$ for $|x|\geq 2R$. Multiply the $j$-th equation of \eqref{syst} by $w_j \zeta_R^2$, integrate by parts in space and in the interval $[t, T]$ in time to get

\begin{multline*}
 \int_t^T\frac d{dt}\int_{\R^d} \frac{w_j^2}2\zeta_R^2 dx ds +  \int_t^T\int_{\R^d}  B
 \cdot w\left(\frac{\de w_j}{\de x_j}\zeta_R^2+w_j\frac{\de \zeta_R^2}{\de x_j}\right) dx ds =
 \\
 \int_t^T\int_{\R^d} \left( |Dw_j|^2\zeta_R^2 +  w_j Dw_j\cdot D\zeta_R^2\right) dx ds + \int_t^T  \int_{\R^d} \bar F
\left(  \frac{\de w_j}{\de x_j}\zeta_R^2 +w_j\frac{\de \zeta_R^2}{\de x_j}\right)dx ds ,
\end{multline*}
so that
 \begin{multline}
 \label{estim}
\int_{\R^d} \frac{w_j^2(t, \cdot)}2\zeta_R^2 dx - \int_{\R^d} \frac{w_j^2(T, \cdot)}2\zeta_R^2 dx + \int_t^T\int_{\R^d} |Dw_j|^2\zeta_R^2 dx ds \leq
\\
\int_t^T\int_{\R^d} |w_j|\left|D w_j\right| 2\zeta_R |D\zeta_R| dx ds +
\|B\|_\infty
\int_t^T\int_{\R^d}\left( |w| \left| \frac{\de w_j}{\de x_j}\right| \zeta_R^2 + |w| | w_j| \left|\frac{\de\zeta_R^2}{\de x_j}\right|\right)  dx ds +
\\
 \int_t^T  \int_{\R^d} |\bar F|
\left(\left|\frac{\de w_j}{\de x_j}\right|  \zeta_R^2  + |w_j| \left|\frac{\de\zeta_R^2}{\de x_j}\right| \right) dx ds .
\end{multline}
We estimate
\[
\int_t^T\int_{\R^d} |w_j|\left|D w_j\right| 2\zeta_R |D\zeta_R|dx ds \leq   
 \eps  \int_t^T\int_{\R^d} |Dw_j|^2\zeta_R^2 dx ds +
\frac 1{\eps}\int_t^T\int_{\R^d} w_j^2 |D\zeta_R|^2 dx ds
\]
and observe that all integrals involving derivatives of $\zeta_R$ vanish as $R\to\infty$ because $D\zeta_R\to 0$ a.e. and $w, \bar F \in 
L^2([0,T]\times \R^d)$. Next we estimate
\[
\int_t^T\int_{\R^d} |w| \left| \frac{\de w_j}{\de x_j}\right| \zeta_R^2 dx ds \leq  \int_t^T\int_{\R^d}\left( \frac 1{2\eps}|w|^2 \zeta_R^2 + \frac \eps 2\left| \frac{\de w_j}{\de x_j}\right|^2 \zeta_R^2 \right)dx ds ,
\]
\[
\int_t^T\int_{\R^d} |\bar F| \left| \frac{\de w_j}{\de x_j}\right| \zeta_R^2 dx ds \leq  \int_t^T\int_{\R^d}\left( \frac 1{2\eps}|\bar F|^2 \zeta_R^2 + \frac \eps 2\left| \frac{\de w_j}{\de x_j}\right|^2 \zeta_R^2 \right)dx ds .
\]
We plug these inequalities in \eqref{estim} with $\eps$ satisfying $1=(\|B\|_\infty+3)\eps/2$, so that the terms involving $Dw_j$ cancel out.
Now we can let 
 $R\to\infty$, sum over $j$, 
and use the terminal condition on $w=Dv$ with the assumption \eqref{Glip} to get
\[
\|w(t,\cdot)\|_2^2 \leq 
L_{G} \|m
(T,\cdot)
\|_2^2 + \int_t^T  \frac d{\eps} \|\bar F(s,\cdot)\|_2^2 ds +  \frac {d\|B\|_\infty}{\eps}\int_t^T\|w(s,\cdot)\|_2^2 ds .
\]
Gronwall inequality gives, for all $0\leq t\leq T$,
\[
\|w(t,\cdot)\|_2^2 \leq \left(
L_{G} \|m(T,\cdot)\|_2^2 + \frac d{\eps} \int_t^T   \|\bar F(s,\cdot)\|_2^2 ds\right)e^{ 
d\|B\|_\infty T/
\eps
} .
\]
We set $c_o:=d
(\|B\|_\infty + 3)/2=d
/\eps$ and use \eqref{Flip} to get
 \begin{equation}
 \label{est_1}
\|Dv(t,\cdot)\|_2^2 \leq \left(
L_{G} \|m(T,\cdot)\|_2^2 + c_o
L_F\int_t^T   \|m(s,\cdot)\|_2^2  ds\right)e^{ c_o
\|B\|_\infty T} .
 \end{equation}
 
 
\noindent {\em Step 3.} Observe that $m$ satisfies 
 \begin{equation*}
\left\{
\begin{array}{lll}
m_t - \text{div}\left(
DH(Dv_1) m\right) = \Delta m +  \text{div}\left((DH(Dv_1)-DH(Dv_2)) m_2\right)
  \quad 
\text{ in  }(0, T)\times \R^d, \\
 m(0,x)= 0 .
 \end{array}
\right.\,
\end{equation*}
Define $\tilde B(t,x):=DH(Dv_1)$, the matrix 
$$A(t,x):= m_2 \int_0^1D^2H(Dv_2+s(Dv_1-Dv_2)) ds$$
 and  $\tilde F(t,x):=A(t,x)(Dv_1-Dv_2)$.  Then the PDE for $m$ reads
\[
m_t - \text{div}(
\tilde B m) = \Delta m +  \text{div}\tilde F ,
\]
with $\tilde B$ and $A$ bounded by the assumption {\em (i)} or {\em (ii)} and the estimate \eqref{est_m_bis}.
As in Step 2 we multiply the equation by $m \zeta_R^2$ and integrate by parts. Now $m \in L^2([0,T]\times \R^d)$ by Step 0 and $\tilde F \in L^2([0,T]\times \R^d)$ by \eqref{est_1}. Then we can estimate  as in Step 2 and let $R\to\infty$ to get
\[
\|m(t,\cdot)\|_2^2 \leq  \frac 1{\eps} \int_0^t   \|\tilde F(s,\cdot)\|_2^2 ds +  \frac {\|\tilde B\|_\infty}{\eps}\int_0^t\|m(s,\cdot)\|_2^2 ds ,
\]
where we used the initial condition $m(0,x)=0$ and chose $\eps=2/(\|\tilde B\|_\infty + 3)=:1/c_1$.
Then Gronwall inequality gives, for all $0\leq t\leq T$,
 \begin{equation}
 \label{est_2}
\|m(t,\cdot)\|_2^2 \leq   c_1 e^{
c_1\|\tilde B\|_\infty T} \int_0^t   \|\tilde F(s,\cdot)\|_2^2 ds \leq  c_1 e^{
c_1\|\tilde B\|_\infty T} \|A\|^2_\infty  \int_0^t   \| Dv(s,\cdot)\|_2^2 ds.
 \end{equation}

\noindent {\em Step 4.} Now we set $\phi(t):= \|Dv(t,\cdot)\|_2^2$ and combine \eqref{est_1} and \eqref{est_2} to get
 \begin{equation}
 \label{phi}
\phi(t)\leq C_1 C_3 \int_0^T \phi(s) ds + C_2 C_3 \int_t^T \int_0^\tau  \phi(s) ds \,d\tau ,
 \end{equation}
for suitable explicit constants $C_i$ depending only on the quantities listed in the statement of the theorem.
Then $\Phi := \sup_{0\leq t\leq T} \phi(t)$ satisfies
\[
\Phi\leq \Phi(TC_1C_3 + T^2C_2C_3/2)
\]
which implies $ \Phi=0$ if $T<\bar T:= (C_1+\sqrt{C_1^2+ 2C_2/C_3})/C_2$. Therefore for such $T$ we conclude that $Dv_1(t,x)=Dv_2(t,x)$ for all $x$ and $0\leq t\leq T$. By the uniqueness of solution for the KFP equation we deduce $m_1=m_2$ and then, by the uniquness of solution of the HJB equation, $v_1=v_2$.
 \end{proof} 
\begin{rem}
\label{plions}
\upshape
The same result holds for solutions $\Z^d$-periodic in the space variable $x$ in the case that $F$ and $G$ are $\Z^d$-periodic in $x$, with the same  proof (and no need of cutoff). In such case of periodic boundary conditions a uniqueness result for short $T$ was presented by Lions in \cite{L} for regularizing running cost $F$ and for terminal cost $G$ independent of $m$. 
 He used estimates in $L^1$ norm for $m$ and in $L^\infty$ norm for $Dv$, instead of the $L^2$ norms we used here in \eqref{est_1} and \eqref{est_2}. Our main contributions are the replacement of the hard estimates stated by Lions with more direct ones obtained by energy methods that require less assumptions on $F$, and the consideration of a 
  cost $G$ depending on the terminal density $m(T)$.
\end{rem}
\begin{rem}\upshape
If the terminal cost $G$ satisfies \eqref{Glip} with $L_G=0$, i.e., $D
G$ does not depend on $m(T)$, then $C_1=0$ and $\bar T=\sqrt{2/(C_2C_3)}$. 
By a simple and elegant argument of Lions \cite{L} based on the inequality \eqref{phi}, the condition for uniqueness $T\leq\bar T$ can be improved in this case to $T\leq \pi/(2\sqrt{C_2C_3})$.
\end{rem}
%
%
%
%
\begin{rem}
\label{C11}
\upshape
The proof of the theorem shows also that there is uniqueness for any $T>0$ if {\em (i)} holds with $\bar C_H$ sufficiently small, or if {\em (ii)} holds with $\sup_{|p|\leq K}|D^2H(p)|$ sufficiently small. In fact,  $\|A\|_\infty$ becomes small as this quantities get small, and so the constant $C_3$ can be made as small as we want. This generalizes a recent result in \cite{T} for nonconvex $H$ with both $C_H$ and $\bar C_H$ small (and periodic boundary conditions).
\end{rem}
\begin{rem}\upshape
\label{C12}
The proof of the theorem shows also the uniqueness of the solution for any $T>0$ if the Lipschitz constants $L_F$ and $L_G$ of the running cost $F$ and of $D_xG$ are small enough, because the constants $C_1$ and $C_2$ can be made small. See also \cite{Amb} for existence and uniqueness results 
under smallness assumptions on the data.
\end{rem}
\begin{rem}\upshape
The estimate \eqref{est_m_bis}, obtained by probabilistic methods in the Appendix, could be replaced by
 \begin{equation*}
 \label{est_m}
  \|m_i(t,\cdot)\|_\infty\leq C_d \left( \|\nu\|_\infty +T(1+\|DH(Dv_i)\|_\infty)^{d+4} \right),  \quad i=1, 2, \;\forall \,t\in[0,T],
 \end{equation*}
 which can be deduced from Corollary 7.3.8 of  \cite{BKRS}, p. 302.
Another alternative estimate, 
obtained by a simple application of the comparison principle, 
 gives 
$$m_i(t,x)\leq e^{t \sup (\text{div} DH(Dv_i))^+} \sup \nu,$$ where, however, the right hand side may depend on the second derivatives of $v_i$. 
\end{rem}
\begin{rem}\label{rmk:refs}
\upshape
A similar structure of proof, based on combining backward and forward integral estimates, was also employed for finite state MFGs by \cite{GMS}.
An existence and uniqueness result under a ''small data'' condition was proved in \cite{HCM:07jssc} for Linear-Quadratic-Gaussian MFGs using a contraction mapping argument to solve the associated system of Riccati differential equations. For different classes of linear-quadratic \mbox{MFGs}, a contraction mapping argument is also used in \cite{wangzhang12, moonbasar16} to obtain existence and uniqueness of solutions under certain implicitly ``small data'' assumptions. The work \cite{ahuja16} mentioned in the introduction provides existence and uniqueness for \mbox{MFGs} with common noise under a weak monotonicity condition in the spirit of Lasry-Lions. The proof relies on the representation of the \mbox{MFG} in terms of a forward-backward stochastic differential equation. Existence of a unique solution is first established for small times by a contraction mapping argument. The monotonicity assumption then allows to extend the solution to any given time horizon. Finally, our assumption $D(v_1-v_2)\in L^2([0,T]\times \R^d)$ can be replaced by $(v_1-v_2)_t\in L^2([0,T]\times \R^d)$, see \cite{Far}.
\end{rem}
\begin{rem} On the smoothness of $H$. \upshape
The assumption $H\in C^2(\R^d)$ can be relaxed to $H\in C^1(\R^d)$  with $DH$ locally Lipschitz. Then the statement of the Theorem remains true with the following changes:  in case {\em (i)} $DH$ is assumed globally Lipschitz and $\bar C_H$ is redefined as its Lipschitz constant, in case {\em (ii)} the time $\bar T$ depends on the Lipschitz constant of $DH$ on the ball $\{p : |p|\leq K\}$ instead of $\sup_{|p|\leq K}|D^2H(p)|$. The proof is the same, after changing Step 3 with the observation that 
 there exists a measurable and locally bounded matrix valued function $\tilde A$ such that 
$$DH(p)-DH(q)=\tilde A(p,q)(p-q),$$
see \cite{BC} for a proof of this fact.
\end{rem}
\begin{ex} 
\label{regul}
Regularizing costs. \upshape
Consider $F$ and $G$ of the form
\[
F(x, \mu)= F_1\left(x, \int_{\R^d} k_1(x,y) \mu(y) dy\right) , \quad G(x, \mu)= 
g_1(x)\int_{\R^d} k_2(x,y) \mu(y) dy +g_2(x)
\]
with $F_1 
 : \R^d \times \R\to \R$ measurable and $k_1, k_2
\in L^2( \R^d \times \R^d)$. Then 
\[
\| \int_{\R^d} k_i(\cdot,y) \mu(y) dy -  \int_{\R^d} k_i(\cdot,y) \nu(y) dy\|_2 \leq \|k_i(\cdot, \cdot)\|_2 \|\mu - \nu\|_2, \quad i=1, 2 .
\]
About $F$ we suppose
$$|F_1(x, r)- F_1(x, s)|\leq L_1
 |r-s|   \quad \forall \, x\in\R^d, r,s\in \R$$
 and get \eqref{Flip} with $L_F=L_1^2 \|k_1(\cdot, \cdot)\|_2^2$. About $G$ we assume $g_1, g_2\in C^1(\R^d)$, $Dg_1$ bounded,  $D_x k_2\in L^2( \R^d \times \R^d)$. Then 
 $$DG(x, \mu)=Dg_1(x)\int_{\R^d} k_2(x,y) \mu(y) dy +g_1(x)\int_{\R^d} D_xk_2(x,y) \mu(y) dy  + Dg_2(x)$$
 satisfies
 \[
 \|DG(\cdot, \mu) - DG(\cdot, \nu)\|_2\leq \left(\|Dg_1\|_\infty \|k_2(\cdot, \cdot)\|_2 
  +\|g_1\|_\infty \|D_xk_2(\cdot, \cdot)\|_2\right) \|\mu - \nu\|_2 ,
 \]
which implies \eqref{Glip}.
\end{ex} 
 \begin{ex} Local costs. \upshape
 Take $G=G(x)$ independent of $m(T)$ and $F$ of the form
 \[
 F(x,\mu)=F_l(x,\mu(x))
 \]
 with $F_l : \R^d\times [0,+\infty)\to\R$ such that
 $$
 |F_l(x, r)- F_l(x, s)|\leq L_l
 |r-s|   \quad \forall \, x\in\R^d, \,r,s\geq 0 .
 $$
  Then $F$ satisfies \eqref{Flip} with $L_F=L_l^2$. 
\end{ex}
\begin{rem}\upshape
The functional $M(\mu)=\int_{\R^d} y\mu(y) dy$ is Lipschitz in $L^2$ among densities of measures with support contained in a given compact set, by Example \ref{regul}. This is the case, for instance, of periodic boundary conditions where the measures are on the torus $\T^d$ (see Remark \ref{plions}). In such a case $F$ of the form $F(x, M(\mu))$ satisfies \eqref{Flip} if it is Lipschitz in the second entry uniformly with respect to $x$.
\end{rem}
\begin{rem}\upshape
If we compare the uniqueness Theorem \ref{st-uniq} with the non-uniqueness Theorem \ref{T1} there are several different assumptions. For instance, in Theorem \ref{T1}  $DH$ is discontinuous in $0$ and $v_1-v_2\notin L^2(\R)$, and  the costs in Example \ref{ese1} do not satisfy \eqref{Flip} and \eqref{Glip}.  So it is not clear which of these conditions is mostly responsible of the lack of uniqueness 
 for short $T$. From the proof of Theorem  \ref{st-uniq} we guess that at least the 
continuity of $DH$ is indispensable for the short-horizon uniqueness. 

\end{rem}
\section{Uniqueness and non-uniqueness for two populations}
\label{two populations}
Consider the system of MFG PDEs corresponding to two populations
\begin{equation}\label{pde2}
\left\{
\begin{array}{lll}
-\de_t v_i+ H_i(\de_x v_i) = \frac 12 \s_i^2(x) \de_{xx}v_i + F_i(x, m_1(t,\cdot),m_2(t,\cdot))  \quad 
 \text{ in  }(0, T)\times \R, \\ \\
 v_i(T,x)=G_i(x, m_1(T),m_2(T)) ,\\ \\
\de_t m_i - (
H_i' (\de_x v_i) m_i)_x = \frac 12 \s_i^2(x) \de_{xx}m_i  \quad 
\text{ in  }(0, T)\times \R, 
\\ \\  m_i(0,x)=\nu_i(x) , 
\qquad\qquad i=1, 2 ,
 \end{array}
\right.\,
\end{equation}
where the Hamiltonians are 
\begin{equation}
\label{H_i}
H_i(p) := \max_{a_i\leq \g \leq b_i}\{-p\g\} =
\left\{
\begin{array}{ll} 
-b_ip &\text{ if } p\leq 0, \\ -a_ip &\text{ if } p\geq 0 ,\end{array} \qquad a_i<0<b_i , 
\right.\,
	\end{equation}
so that
\[
H_i'(p) = -b_i \text{ if } p<0 , \quad H_i'(p) = -a_i \text{ if } p>0 ,
\]
$\sigma_i
$ satisfy \eqref{sigma}, 
and  $F_i, G_i : \R\times \tilde{\mathcal{P}}_{1}(\mathbb{R})\times \tilde{\mathcal{P}}_{1}(\mathbb{R}) \to \R$ 
verify the same regularity conditions as $F, G$ in Section \ref{section pde}.

We give two different sets of qualitative assumptions that produce examples of nonuniqueness. The first is the following, for $i=1, 2$,
\begin{equation}\label{ass-syst-F-1}
	D_xF_i
	(x,\mu_1,\mu_2)
	\left\{
	\begin{array}{ll}\leq 0  
	\quad &\text{ if } M(\mu_1), M(\mu_2)>
	 0 ,\\ 
	\geq 0  
	\quad &\text{ if } M(\mu_1), M(\mu_2)<
	 0 .\end{array}
\right.\,
	\end{equation}

\begin{equation}\label{ass-syst-G-1}
	D_x G_i
	(x,\mu_1,\mu_2)
	\left\{
	\begin{array}{ll}\leq 0  \quad\mbox{ and }\not\equiv  0  \quad &\text{ if } M(\mu_1), M(\mu_2)>
	 0 ,\\ 
	\geq 0  \quad\mbox{ and }\not\equiv  0  \quad &\text{ if } M(\mu_1), M(\mu_2)<
	 0 .\end{array}
\right.\,
		\end{equation}
 \begin{prop}
 \label{nonex5}
Assume \eqref{H_i}, $F_i$ satisfy {\rm (F1)} and \eqref{ass-syst-F-1}, $G_i$ satisfy {\rm (G1)} and \eqref{ass-syst-G-1}.
Then, for all $\nu_i$ with $M(\nu_i)=0$, there is a classical solution of \eqref{pde2} with $\de_x v_1(t,x), \de_x v_2(t,x)<0$ for all $0<t<T$ and a classical solution with $\de_x v_1(t,x), \de_x v_2(t,x)>0$ for all $0<t<T$.
\end{prop}
   \begin{proof} 
   For the solution with $\de_x v_1(t,x), \de_x v_2(t,x)<0$ we make the ansatz that $m_i$ solves 
  \begin{equation}
  \label{m_i_eq}
   \de_t m_i + b_i \de_x m_i = \frac 12 \s_i^2(x) \de_{xx}m_i  \quad 
\text{ in  }(0, T)\times \R, \quad\
  m_i(0,x)=\nu_i(x) ,
	\end{equation}
 so that   $M(m_1(t,\cdot))=b_1t>0$ and $M(m_2(t,\cdot))=b_2 t>0$.
We solve the equations 
  \begin{equation}
  \label{v_i_eq}
-\de_t v_i - b_i \de_x v_i = \frac 12 \s_i^2(x) \de_{xx}v_i + F_i(x, m_1(t,\cdot),m_2(t,\cdot)) , \;\;
 v_i(T,x)=G_i(x, m_1(T),m_2(T)) ,
	\end{equation}
for $i=1, 2$. Then $w_i:=\de_x v_i$ solves
    \begin{equation}
  \label{w_i_eq}
 -\de_t w_i -b_i \de_x w_i -
 \s_i(x) \s_i(x)_x  \de_x w_{i} - \frac 12 \s_i^2(x)  \de_{xx}w_{i} = D_x F_i
 (x, m_1(t,\cdot),m_2(t,\cdot))\leq 0 , 
 \text{ in  }(0, T)\times \R , 
 	\end{equation}
     \begin{equation}
  \label{w_i_tcond} 
 w_i(T,x)=D_x G_i
 (x, m_1(T,\cdot),m_2(T,\cdot))<0 ,
 \end{equation}
and so $w_i<0$ for all $t<T$ by the Strong Maximum Principle. On the other hand, if $\de_x v_i<0$ the equations for $m_i$ in \eqref{pde2}  are 
 \eqref{m_i_eq} and the equations for $v_i$ in \eqref{pde2}  are 
 \eqref{v_i_eq}, as we guessed. 
 
 The construction of a solution with $\de_x v_1(t,x), \de_x v_2(t,x)>0$ is done in a symmetric way, starting with the ansatz that $m_i$ solves 
\begin{equation}
\label{m_i_eq2}
   \de_t m_i +a_i \de_x m_i = \frac 12 \s_i^2(x) \de_{xx}m_i  \quad 
\text{ in  }(0, T)\times \R, \quad\
  m_i(0,x)=\nu_i(x) ,
	\end{equation}
so now   $M(m_1(t,\cdot))=a_1t<0$ and $M(m_2(t,\cdot))=a_2 t<0$. We proceed as before by solving a linear equation  for $v_i$ like \eqref{v_i_eq} but with the term $b_i \de_x v_i$ replaced by $a_i \de_x v_i$. Now $D_x F_i
\geq 0$ in the equation for $w_i:=\de_x v_i$ and $D_x G_i
>0$ in the terminal conditions, so $w_i>0$ for all $t<T$.
\end{proof}
The second set of assumptions is the following.
\begin{equation}\label{ass-syst-F-2}
	D_x F_1
	(x,\mu_1,\mu_2)
	\left\{
	\begin{array}{ll}\leq 0  
	\quad &\text{ if } M(\mu_1)> 0 , M(\mu_2)<0 \\ 
	\geq 0  
	\quad &\text{ if } M(\mu_1)< 0 , M(\mu_2)>0.\end{array}
\right.\,
	\end{equation}
\begin{equation}\label{ass-syst-F-3}
	D_x F_2
	(x,\mu_1,\mu_2)
	\left\{
	\begin{array}{ll}\leq 0  
	\quad &\text{ if }  M(\mu_1)< 0 , M(\mu_2)>0 ,\\ 
	\geq 0  
	\quad &\text{ if }  M(\mu_1)> 0 , M(\mu_2)<0 .\end{array}
\right.\,
	\end{equation}
\begin{equation}\label{ass-syst-G-2}
	D_x G_1
	(x,\mu_1,\mu_2)
	\left\{
	\begin{array}{ll}\leq 0  \quad\mbox{ and }\not\equiv  0  \quad &\text{ if }M(\mu_1)> 0 , M(\mu_2)<0 ,\\ 
	\geq 0  \quad\mbox{ and }\not\equiv  0  \quad &\text{ if } M(\mu_1)< 0 , M(\mu_2)>0
	 .\end{array}
\right.\,
		\end{equation}
\begin{equation}\label{ass-syst-G-3}
	D_x G_2
	(x,\mu_1,\mu_2)
	\left\{
	\begin{array}{ll}\leq 0  \quad\mbox{ and }\not\equiv  0  \quad &\text{ if } M(\mu_1)< 0 , M(\mu_2)>0 ,\\ 
	\geq 0  \quad\mbox{ and }\not\equiv  0  \quad &\text{ if }  M(\mu_1)> 0 , M(\mu_2)<0 
	.\end{array}
\right.\,
		\end{equation}
 \begin{prop} \label{nonex6}
 Assume \eqref{H_i}, $F_i$ satisfy {\rm (F1)}, \eqref{ass-syst-F-2},  and \eqref{ass-syst-F-3}, $G_i$ satisfy {\rm (G1)}, \eqref{ass-syst-G-2}, and \eqref{ass-syst-G-3},  $i=1, 2$. 
 Then, for all $\nu_i$ with $M(\nu_i)=0$, there is a solution of \eqref{pde} with $\de_x v_1(t,x)<0, \de_x v_2(t,x)>0$ for all $0<t<T$ and a solution with $\de_x v_1(t,x)>0, \de_x v_2(t,x)<0$ for all $0<t<T$.
\end{prop} 
   \begin{proof} 
   This is a variant of the preceding proof, so we only explain the changes. For the solution with $\de_x v_1
   <0, \de_x v_2
   >0$ we start solving \eqref{m_i_eq} for $i=1$ and \eqref{m_i_eq2} for $i=2$, so that $M(m_1(t
   ))>0$ and $M(m_2(t
   ))
   <0$. Then for $i=1$ we solve \eqref{v_i_eq} and get  \eqref{w_i_eq} and \eqref{w_i_tcond}, by the assumptions  \eqref{ass-syst-F-2} and  \eqref{ass-syst-G-2}, so $\de_x v_1
   <0$.   For $i=2$ we solve \eqref{v_i_eq}  with the term $b_i \de_x v_i$ replaced by $a_2 \de_x v_2$. Then $D_x F_2
\geq 0$ in the equation for $w_2:=\de_x v_2$ and $D_x G_2
>0$ in the terminal conditions,  by the assumptions  \eqref{ass-syst-F-3} and  \eqref{ass-syst-G-3}. 
 Thus $\de_x v_2>0$. 
 The construction of the solution with $\de_x v_1
   >0, \de_x v_2
   <0$ is symmetric.
\end{proof}

\begin{rem}
\upshape
In Proposition \ref{nonex5} the assumption $D_xG_i\not\equiv 0$ in \eqref{ass-syst-G-1} can be dropped and replaced by $D_xF_i\not\equiv 0$ in \eqref{ass-syst-G-1}, by the argument of Theorem \ref{T1}. The same variant can be done on the assumptions of Proposition \ref{nonex6}.
\end{rem}
%
%
Next we use the last two propositions and the example to check the sharpness of some sufficient conditions for uniqueness. A natural generalization to systems with two populations of the 
monotonicity conditions \eqref{monoG}  \eqref{monoF} 
 is given the following result, see also \cite{Cir} for stationary equations.
\begin{thm} 
\label{uniq-th}
Assume $H_i$ are convex and Lipschitz, $\s_i>0$ are constant, $\nu_i \,dx\in  \mathcal{P}_{1}(\mathbb{R})$, the functions $x\mapsto F_i(x,\mu_1, \mu_2)$ and $x\mapsto G_i(x,\mu_1, \mu_2)$ 
depend only on $M(\mu_1), M(\mu_2)$, they grow at most linearly in $x$ for bounded  $\int_\R |x| d\mu_i(x)$, and for some $\lambda_i >0$ and all $(\mu_1, \mu_2),(\bar \mu_1, \bar \mu_2) \in \mathcal{P}_{1}(\mathbb{R})^2$ with a $C^2$ density
\begin{equation}\label{mono2} 
  \int_\R \sum_{i=1}^2 \l_i[F_i(x,\mu_1,\mu_2) - F_i(x,\bar \mu_1,\bar \mu_2)] d(\mu_i -\bar \mu_i)(x) 
  > 0 , \; \text{if } M(\mu_1)\ne  M(\bar \mu_1) \text{ or }  M(\mu_2)\ne  M(\bar \mu_2) ,
  \end{equation}
  \begin{equation}\label{mono3}
  \int_\R \sum_{i=1}^2 \l_i[G_i(x,\mu_1,\mu_2) - G_i (x,\bar \mu_1,\bar \mu_2)] d(\mu_i -\bar \mu_i)(x) \geq 0 . 
  \end{equation}
Then there is at most one classical solution $(v_1, v_2, m_1, m_2)$ of the problem \eqref{pde2} such that each $v_i$ and its derivatives are bounded by $C(1+|x|)$.
\end{thm}
   \begin{proof} We follow 
   the proof of Theorem \ref{th:uniq}. Let  $(\bar v_1, \bar  v_2, \bar m_1,\bar  m_2)$ be a second solution. We multiply the  equations of the $i$-th population 
    by $\l_i$ and add them over $i$. After using the convexity of $H_i$, as in  the proof of Theorem \ref{th:uniq}, we reach 
 \begin{multline*}
\int_\R \sum_{i=1}^2 \l_i[G_i(x,m_1(T),m_2(T)) - G_i(x,\bar m_1(T),\bar m_2(T))] d(m_i(T) -\bar m_i(T))(x) +\\
 \int_0^T \int_\R \sum_{i=1}^2 \l_i[F_i(x,m_1(t),m_2(t)) - F_i(x,\bar m_1(t),\bar m_2(t))] d(m_i(t) -\bar m_i(t))(x) \,dt
 \leq 0 .
\end{multline*}
Then \eqref{mono2} and \eqref{mono3} imply $M(m_1(t)) = M(\bar m_1(t))$ and $M(m_2(t)) = M(\bar m_2(t))$ for all $t$. Since $F_i$ and $G_i$ depend only on $M(\mu_1), M(\mu_2)$, from the HJB equations we get $v_i=\bar v_i$ and finally the KFP equations give $m_i=\bar m_i$, $i=1, 2$.
\end{proof}
We consider the following example
	\begin{equation}\label{Fi-lin}
	F_i(x,\mu_1, \mu_2)=\a_i x M(\mu_1) + \beta_i x M(\mu_2) + f_i(\mu_1,\mu_2) , \quad i=1, 2 ,
		\end{equation}
	\begin{equation}\label{Gi-lin}
	 G_i(x,\mu_1,\mu_2)=\g_i x M(\mu_1) + \d_i x M(\mu_2) + g_i(\mu_1,\mu_2) ,  \quad i=1, 2 ,
		\end{equation}
	with $\alpha_i, \beta_i, \g_i, \d_i\in \R$, $f_i, g_i : 
	{\mathcal{P}}_{1}(\mathbb{R})^2 \to \R$. 

\begin{cor} 
\label{cor-linear}
Let $H_i$ be given by \eqref{H_i} and $F_i, G_i$ of the form \eqref{Fi-lin},  \eqref{Gi-lin} with  $f_i, g_i$ 
 $d_1$-continuous. 
Then  the problem \eqref{pde2} has at most one solution with 
derivatives bounded by $C(1+|x|)$, if $\nu_i \,dx\in  \mathcal{P}_{1}(\mathbb{R})$ and  there exists $\lambda >0$ such that the matrices
  $$M_1:= \left( \begin{array}{cc}
\l\a_1 & \l \b_1 \\
\a_2 &  \b_2\end{array} \right) ,
\qquad M_2:=  \left( \begin{array}{cc}
\l\g_1 & \l \d_1 \\
\g_2 &  \d_2\end{array} \right)$$ 
are, respectively,  positive definite and positive semi-definite;
\noindent
on the other hand,
it has at least two solutions with derivatives bounded  
 by $C(1+|x|)$ if $\nu_i \in \tilde{ \mathcal{P}}_1(\R)$, $M(\nu_i)=0$, and either
 \begin{equation}
 \label{ex-1nonuniq}
	\a_i, \b_i, \g_i, \d_i \leq 0 , \quad \g_i+\d_i<0 ,  \quad i=1, 2 ,
  \end{equation}
	or
	 \begin{equation}
 \label{ex-2nonuniq}
	\a_1, \b_2, \g_1, \d_2 \leq 0 , \quad \a_2, \b_1, \g_2, \d_1 \geq 0 , \quad \g_1<\d_1 , \quad \g_2>\d_2.
	  \end{equation}
\end{cor}
   \begin{proof} 
   It is easy to compute the integral in \eqref{mono2} and get, after normalizing $\l_2$ to 1 and setting $\l_1=\l$,
   \[
   \l\a_1(M(m_1)-M(\bar m_1))^2 + (\l\b_1+\a_2)(M(m_1)-M(\bar m_1))(M(m_2)-M(\bar m_2)) +\b_2(M(m_2)-M(\bar m_2))^2,
   \]
   so condition  \eqref{mono2} is satisfied if the matrix 
  $M_1$ %
is positive definite.  
Similarly, condition  \eqref{mono3} is satisfied if the quadratic form associated to $M_2$ 
is positive semi-definite. 

As for  non-uniqueness,  the first statement follows from Prop. \ref{nonex5}  and the second from Prop. \ref{nonex6}. 
\end{proof}
\begin{rem}  \upshape
The conditions for the definiteness of $M_1$
 \begin{equation}
  \label{ex-uniq}
 \a_1>0 ,  \qquad \b_2>0 ,  \qquad \l\a_1\b_2>(\l\b_1+\a_2)^2/4 ,
  \end{equation}
require not only the form of crowd-aversion within each population 
 explained in Remark \ref{crowd}  
for the case of a single population, but also that the costs for intraspecific interactions are dominant over the costs of the interactions of a population with the other. The same holds for the conditions of semi-definiteness of $M_2$, i.e.,
\[
\g_1\geq 0 , \qquad \d_2 \geq 0, \quad 
\qquad   \l\g_1\d_2\geq(\l\d_1+\g_2)^2/4 .
\]
On the other hand, the hypotheses of the examples of non-uniqueness \eqref{ex-1nonuniq} or \eqref{ex-2nonuniq} hold only if the intraspecific costs are 
null or imitation is rewarding within each population.

Note also that the gap between the sufficient conditions for uniqueness and for non-uniqueness is larger here than in the case of a single population.
\end{rem}

\begin{ex}
\upshape
Consider a terminal cost $G$ as in Corollary \ref{cor-linear} and satisfying
\[
\g_1, \d_2 \geq 0 , \qquad 
 \d_1=\g_2=0 ,
\]
so that, in particular, there is no cost or gain for interspecific interactions at the terminal time $T$.
Then the sufficient condition for uniqueness reduces to \eqref{ex-uniq} for some $\l >0$, 
and this holds under the simple conditions 
\begin{equation}
\label{uniq-F}
\a_1, \b_2 >0 , \quad \a_1\b_2 >\b_1 \a_2 ,
\end{equation}
as it can be easily seen by choosing $\l=\frac{2\a_1\b_2 - \b_1\a_2}{\b_1^2}$ if $\b_1\ne0$, and $\l=\frac{2\a_1\b_2 - \b_1\a_2}{\a_2^2}$ if $\a_2\ne0$.
\end{ex}
\begin{rem} 
 \upshape
If Theorem \ref{uniq-th} is specialized to stationary equations, it improves slightly upon the uniqueness result in \cite{Cir} because it allows to choose the parameters $\l_i$. This can be seen in the case of local and linear costs
\[
 F_i(x,\mu_1,\mu_2)=\a_i  \mu_1(x) + \beta_i  \mu_2(x)
 , \quad i=1, 2 ,
	\]
	In fact, the integrand of  the integral in \eqref{mono2} is
	\[
   \l\a_1(m_1 - \bar m_1)(x)^2 + (\l\b_1+\a_2)(m_1 - \bar m_1)(x)(m_2 - \bar m_2)(x) +\b_2(m_2 - \bar m_2)(x)^2,
   \]
and then the condition  \eqref{mono2} is satisfied again if the matrix $M_1$ is positive definite. 
Hence \eqref{uniq-F} is 
 a sufficient condition for uniqueness of the stationary MFG equations for this local case, more general than the assumption in  \cite{Cir}.
\end{rem}
\begin{rem} 
 \upshape
A uniqueness result with assumptions of short time horizon and smooth Hamiltonian replacing the convexity of $H$ and monotonicity of the costs, similar to Theorem \ref{st-uniq}, can be proved also for systems with several population such as \eqref{pde2}. This 
is done in \cite{BC} for Neumann boundary conditions in bounded domains.
\end{rem}
%

\begin{appendix}

\section{Density estimate}

Let $\sigma > 0$. Let $b\!: [0,T]\times \mathbb{R}^{d}\rightarrow \mathbb{R}^{d}$ be bounded and measurable. Let $\nu\in \mathcal{P}(\mathbb{R}^{d})$ be such that $\nu$ is absolutely continuous with respect to Lebesgue measure with bounded density, that is,
\[¤
	\frac{d\nu}{d\lambda_{d}}(.) = m_{0}(.)
\]
for some bounded and measurable $m_{0}\!: \mathbb{R}^{d}\rightarrow [0,\infty)$ with $\int m_{0}(x)dx = 1$.

Let $(\Omega,\mathcal{F},(\mathcal{F}_{t}),\Prb)$ be a standard filtered probability space carrying a $d$-dimensional $(\mathcal{F}_{t})$-Wiener process, an $\mathbb{R}^{d}$-valued $\mathcal{F}_{0}$-measurable random variable $\xi$ with distribution $\Prb\circ (\xi)^{-1} = \nu$ and an $\mathbb{R}^{d}$-valued continuous $(\mathcal{F}_{t})$-adapted process $X$ such that 
\begin{equation} \label{AppEqSDE}
	X(t) = \xi + \int_{0}^{t} b\left(s,X(s)\right)ds + \sigma W(t),\quad t\in [0,T].
\end{equation}
We also assume---as we may---that $(\Omega,\mathcal{F})$ is a Borel space.

Let $\mathfrak{m}$ be the flow of marginal distributions of $X$:
\[
	\mathfrak{m}(t)\doteq \Prb\circ (X(t))^{-1},\quad t\in [0,T].
\]
Lastly, for $t\in (0,T]$, let $p_{t}$ denote the density of the $d$-variate Gaussian distribution with mean zero and covariance matrix $t \Id_{d}$:
\[
	p_{t}(y) \doteq (2\pi t)^{-d/2} e^{-\frac{|y|^{2}}{2t}},\quad y\in \mathbb{R}^{d}.
\]

For the following $L^{\infty}$-estimate on the Lebesgue densities of the flow of measures  $\mathfrak{m}$, we use the Girsanov transformation and a conditioning argument in the spirit of Exercise~7.4 in \cite[p.\,170]{friedman75}. In \cite{qianetal03}, sharp estimates on the transition probability densities are obtained through more sophisticated  probabilistic methods.

\begin{prop}
\label{prop_app}
Under the above assumptions, the marginal distribution of $X$ at any time $t\in [0,T]$ is absolutely continuous with respect to Lebesgue measure with bounded density. More precisely, for every $t\in (0,T]$, there exists a bounded measurable function $m_{t}\!: \mathbb{R}^{d}\rightarrow [0,\infty)$ with $\int m_{t}(x)dx = 1$ such that
\begin{align*}
	& \frac{d\mathfrak{m}(t)}{d\lambda_{d}}(.) = m_{t}(.) & &\text{and}& \| m_{t} \|_{\infty} \leq \hat{C}_{t} \cdot \|m_{0}\|_{\infty},
\end{align*}
where $\hat{C}_{t} = \hat{C}_{t}(\sigma,\|b\|_{\infty}, d)$ is a finite constant that 
 need not be greater than
\[
	{ \left(e^{\frac{8 \|b\|_{\infty}^{2}}{\sigma^{2}}t} + \frac{4\|b\|_{\infty}}{\sigma} \sqrt{2\pi t}\cdot e^{\frac{16 \|b\|_{\infty}^{2}}{\sigma^{2}}t} \right)^{d/4} } \cdot e^{\frac{\|b\|_{\infty}^{2}}{2\sigma^{2}} t} \cdot \int_{\mathbb{R}^{d}} e^{\frac{\|b\|_{\infty}}{\sigma^2}|x|} p_{\sigma^{2} t}(x) dx.
\] 

\end{prop}

\begin{proof}
For $x\in \mathbb{R}^{d}$, define the process $Y^{x} = (Y^{x}_{1},\ldots,Y^{x}_{d})$ through
\[
	Y^{x}(t) \doteq x + \sigma W(t),\quad t\in [0,T],
\]
and a process $Z^{x}$ over $[0,T]$ by
\[
\begin{split}
	Z^{x}(t) &\doteq \exp\left(\frac{1}{\sigma} \int_{0}^{t} b\left(s, Y^{x}(s) \right)\cdot dW(s) - \frac{1}{2\sigma^{2}} \int_{0}^{t} \left| b\left(s, Y^{x}(s) \right) \right|^{2} ds \right) \\
	&= \exp\left(\frac{1}{\sigma} \sum_{i=1}^{d} \int_{0}^{t} b_{i}\left(s, Y^{x}(s) \right) dW_{i}(s) - \frac{1}{2\sigma^{2}} \int_{0}^{t} \sum_{i=1}^{d} \left| b_{i}\left(s, Y^{x}(s) \right) \right|^{2} ds \right).
\end{split}
\]

Girsanov's theorem (for instance, Chapter~7 in \cite{friedman75} or Sections 3.5 and 5.3.B in \cite{karatzasshreve91}) and the fact that $\xi$ and $W$ are independent yield, for $t\in [0,T]$, every bounded measurable function $g\!: \mathbb{R}^{d}\rightarrow \mathbb{R}$,
\[
	\Mean\left[ g\left( X(t) \right) \right] = \int_{\mathbb{R}^{d}} \Mean\left[ g\left( Y^{x}(t) \right) Z^{x}(t)  \right] \nu(dx).
\]
This implies (taking indicator functions for $g$) that
\begin{equation} \label{AppEqGirsanov}
	\mathfrak{m}(t,B) = \int_{\mathbb{R}^{d}} \Mean\left[ \mathbf{1}_{B}\left( Y^{x}(t) \right) Z^{x}(t)  \right] \nu(dx) \text{ for all } B\in \mathcal{B}(\mathbb{R}^d).
\end{equation}
By construction, $\mathfrak{m}(0) = \nu$. Fix $t\in (0,T]$.

For $y\in \mathbb{R}^{d}$ define a process $\tilde{W}^{t,y}$ over $[0,t]$ by
\[
	\tilde{W}^{t,y}(s) \doteq \begin{cases}
		\frac{s}{t}y + (t-s) \int_{0}^{s} \frac{1}{t-r} dW(r) &\text{if } s\in [0,t) \\
		y &\text{if } s = t.
	\end{cases}
\]
Then $\tilde{W}^{t,y}$ is a $d$-dimensional Brownian bridge (with respect to $\Prb$) from $0$ to $y$ over time interval $[0,t]$. The process $\tilde{W}^{t,y}$ is continuous on $[0,t]$ $\Prb$-almost surely; cf.\ Section~5.6.B in \cite[pp.\,358-360]{karatzasshreve91}; its It{\^o} differential is given by
\begin{equation} \label{AppBridgeDifferential}
	d\tilde{W}^{t,y}(s) = \frac{y}{t}ds - \left(\int_{0}^{s} \frac{1}{t-r}dW(r) \right)ds + dW(s),\quad s\in [0,t].
\end{equation}
That \eqref{AppBridgeDifferential} holds if $s < t$ is clear by It{\^o}'s formula and the definition of $\tilde{W}^{t,y}$. To see that it also holds if $s = t$, apply Fubini's theorem, then the Cauchy-Schwarz inequality followed by the It{\^o} isometry, to find that
\begin{multline*}
	\Mean\left[ \int_{0}^{t} \left|\int_{0}^{s} \frac{1}{t-r} dW(r) \right| ds \right] = \int_{0}^{t} \Mean\left[ \left|\int_{0}^{s} \frac{1}{t-r} dW(r) \right| \right] ds \\
	\leq \int_{0}^{t} \sqrt{\Mean\left[ \int_{0}^{s} \frac{dr}{(t-r)^2} \right]} ds = \int_{0}^{t} \sqrt{\frac{s}{t(t-s)}}ds \leq \int_{0}^{t} \frac{ds}{\sqrt{t-s}} < \infty.
\end{multline*}
Using also the continuity of $\tilde{W}^{t,y}$, it follows that with $\Prb$-probability one, for all $t_{1}, t_{2} \in [0,t]$ such that $t_{1} \leq t_{2}$,
\[
	\tilde{W}^{t,y}(t_{2}) - \tilde{W}^{t,y}(t_{1}) = \frac{t_{2}-t_{1}}{t}y - \int_{t_{1}}^{t_{2}} \left(\int_{0}^{s} \frac{1}{t-r}dW(r) \right)ds + W(t_{2}) - W(t_{1}),
\]
which is equivalent to \eqref{AppBridgeDifferential}. In particular, $\tilde{W}^{t,y}$ is a vector of continuous semimartingales on $[0,t]$, with the same cross-variation processes as $W$. Moreover (for instance, Exercise 5.6.17 in \cite[p.\,361]{karatzasshreve91}), we have for all $\hat{g}\!: \mathbf{C}([0,t],\mathbb{R}^{d}) \rightarrow \mathbb{R}$ bounded and measurable,
\begin{equation} \label{AppBMBridge}
	\Mean\left[ \hat{g}(W) \right] = \int_{\mathbb{R}^{d}} \Mean\left[ \hat{g}(\tilde{W}^{t,y}) \right] p_{t}(y) dy,
\end{equation}
where $p_{t}$ is the density of the $d$-variate Gaussian distribution with mean zero and covariance matrix $t \Id_{d}$. Formula \eqref{AppBMBridge} corresponds to conditioning the Wiener process $W$ on its values at time $t$. To be more precise, we choose a regular conditional distribution of $\Prb$ given $W(t)$; cf.\ Theorem~6.3 in \cite[p.\,107]{kallenberg01}. Since $(\Omega,\mathcal{F})$ is Borel, there exists a probability kernel $\kappa_{t}\!: \mathbb{R}^{d}\times \mathcal{F} \rightarrow [0,1]$ such that for every $A\in \mathcal{F}$,
\[
	\kappa_{t}(W(t),A) = \Mean\left[ \mathbf{1}_{A} \,\big|\, W(t) \right] \quad \Prb\text{-almost surely.}
\]
The probability measures $\kappa_{t}(y,\cdot)$, $y\in \mathbb{R}^{d}$, are uniquely determined Lebesgue almost everywhere. In view of \eqref{AppBMBridge}, we have
\begin{equation} \label{AppBMBridgeKernel}
	\kappa_{t}(y,\cdot)\circ (W)^{-1} = \Prb\circ (\tilde{W}^{t,y})^{-1} \text{ for Lebesgue almost every } y\in \mathbb{R}^{d}.
\end{equation}

We are going to apply the above regular conditional distribution to representation \eqref{AppEqGirsanov}. By construction, for all $B\in \mathcal{B}(\mathbb{R}^{d})$, $x\in \mathbb{R}^{d}$,
\begin{multline*}
	\Mean\left[ \mathbf{1}_{B}\left( Y^{x}(t) \right) Z^{x}(t)  \right] = \Mean\left[ \mathbf{1}_{B}\left( x + \sigma W(t) \right) \exp\left(\frac{1}{\sigma} \int_{0}^{t} b\left(s, x + \sigma W(s) \right)\cdot dW(s) \right.\right.\\
	\left.\left. - \frac{1}{2\sigma^{2}} \int_{0}^{t} \left| b\left(s, x + \sigma W(s) \right) \right|^{2} ds \right)  \right].
\end{multline*}
Recalling that $p_{t}$ is the density of the law of $W(t)$, we find that
\[
	\Mean\left[ \mathbf{1}_{B}\left( Y^{x}(t) \right) Z^{x}(t)  \right] = \int_{\mathbb{R}^{d}} \Mean_{\kappa_{t}(y,\cdot)}\left[ \mathbf{1}_{B}\left( Y^{x}(t) \right) Z^{x}(t)  \right] p_{t}(y) dy.
\]
Setting
\[
	\Psi_{t}(x,y) \doteq \Mean\left[ \exp\left(\frac{1}{\sigma} \int_{0}^{t} b\left(s, x + \sigma \tilde{W}^{t,y}(s) \right)\cdot d\tilde{W}^{t,y}(s)
	 - \frac{1}{2\sigma^{2}} \int_{0}^{t} \left| b\left(s, x + \sigma \tilde{W}^{t,y}(s) \right) \right|^{2} ds \right)  \right]
\]
we have, by \eqref{AppBMBridgeKernel},
\begin{equation} \label{AppGirsanovBridgex}
	\Mean\left[ \mathbf{1}_{B}\left( Y^{x}(t) \right) Z^{x}(t)  \right] = \int_{\mathbb{R}^{d}} \mathbf{1}_{B}\left( x + \sigma y \right) \Psi_{t}(x,y) p_{t}(y) dy.
\end{equation}

Thanks to Eq.~\eqref{AppBridgeDifferential}, $\Psi_{t}(x,y)$ can be expressed as
\[
\begin{split}
	\Psi_{t}(x,y) &= \Mean\left[ \exp\left(\frac{1}{\sigma t} \int_{0}^{t} y\cdot b\left(r, x + \sigma \tilde{W}^{t,y}(s) \right)ds \right) \right.\\
	&\cdot \exp\left( -\frac{1}{\sigma} \int_{0}^{t} b\left(r, x + \sigma \tilde{W}^{t,y}(s) \right)\cdot \left( \int_{0}^{s} \frac{1}{t-r} dW(r)\right) ds \right) \\
	&\cdot \left. 
	\exp\left(\frac{1}{\sigma} \int_{0}^{t} b\left(r, x + \sigma \tilde{W}^{t,y}(s) \right)\cdot dW(s) - \frac{1}{2\sigma^{2}} \int_{0}^{t} \left| b\left(r, x + \sigma \tilde{W}^{t,y}(s) \right) \right|^{2} ds 
	 \right) \right].
\end{split}
\]
Below, we will show that for every $c\in [0,\infty)$,
\begin{equation} \label{AppBridgeInt}
	\Mean\left[ \exp\left(c \sum_{i=1}^{d} \int_{0}^{t} \left|\int_{0}^{s} \frac{1}{t-r} dW_{i}(r) \right| ds \right) \right] \leq \left(e^{2c^{2}t} + 2c\sqrt{2\pi t}\cdot e^{4c^{2} t} \right)^{d/2} < \infty.
\end{equation}
Thus, $\Psi_{t}$ is well defined as a (measurable) function $\mathbb{R}^{d}\times \mathbb{R}^{d} \rightarrow (0,\infty)$. By inequality \eqref{AppBridgeInt}, the boundedness of $b$, the Cauchy-Schwarz inequality and the (super-)martingale property of the stochastic exponential of a martingale, we obtain
\begin{align*}
\begin{split}
	\Psi_{t}(x,y) &\leq e^{\frac{\|b\|_{\infty}}{\sigma}|y|} \cdot \Mean\left[ \exp\left( \frac{2\|b\|_{\infty}}{\sigma} \sum_{i=1}^{d} \int_{0}^{t} \left| \int_{0}^{s} \frac{1}{t-r} dW_{i}(r)\right| ds \right) \right]^{1/2} \\
	&\cdot \Mean\left[ e^{ \frac{1}{\sigma} \int_{0}^{t} 2b\left(r, x + \sigma \tilde{W}^{t,y}(s) \right)\cdot dW(s) - \frac{1}{4\sigma^{2}} \int_{0}^{t} \left| 2b\left(r, x + \sigma \tilde{W}^{t,y}(s) \right) \right|^{2} ds}  \right]^{1/2}.
\end{split}\\
\begin{split}
	&\leq e^{\frac{\|b\|_{\infty}}{\sigma}|y|} \cdot \left(e^{\frac{8 \|b\|_{\infty}^{2}}{\sigma^{2}}t} + \frac{4\|b\|_{\infty}}{\sigma} \sqrt{2\pi t}\cdot e^{\frac{16 \|b\|_{\infty}^{2}}{\sigma^{2}}t} \right)^{d/4} \\
	&\cdot \underbrace{\Mean\left[ e^{ \frac{1}{\sigma} \int_{0}^{t} 2b\left(r, x + \sigma \tilde{W}^{t,y}(s) \right)\cdot dW(s) - \frac{1}{2\sigma^{2}} \int_{0}^{t} \left| 2b\left(r, x + \sigma \tilde{W}^{t,y}(s) \right) \right|^{2} ds}  \right]^{1/2} }_{=1} \cdot e^{\frac{\|b\|_{\infty}^{2}}{2\sigma^{2}} t},
\end{split}
\end{align*}
hence
\begin{equation} \label{AppPsiEstimate}
	\sup_{x\in \mathbb{R}^{d}} \Psi_{t}(x,y) \leq C_{t}\cdot e^{\frac{\|b\|_{\infty}}{\sigma}|y|},\quad y\in \mathbb{R}^{d},
\end{equation}
where the finite constant $C_{t} = C_{t}(\sigma,\|b\|_{\infty},d)$ is given by
\[
	C_{t}(\sigma,\|b\|_{\infty}, d)\doteq \left(e^{\frac{8 \|b\|_{\infty}^{2}}{\sigma^{2}}t} + \frac{4\|b\|_{\infty}}{\sigma} \sqrt{2\pi t}\cdot e^{\frac{16 \|b\|_{\infty}^{2}}{\sigma^{2}}t} \right)^{d/4} \cdot e^{\frac{\|b\|_{\infty}^{2}}{2\sigma^{2}} t}.
\]

Recalling \eqref{AppEqGirsanov}, \eqref{AppGirsanovBridgex} and the hypothesis that $\nu$ has density $m_{0}$, we see that for all $B\in \mathcal{B}(\mathbb{R^{d}})$,
\begin{align*}
	\mathfrak{m}(t,B) &= \int_{\mathbb{R}^{d}} \int_{\mathbb{R}^{d}} \mathbf{1}_{B}\left( x + \sigma y \right) \Psi_{t}(x,y) p_{t}(y) m_{0}(x) dxdy \\
	&= \int_{\mathbb{R}^{d}} \int_{\mathbb{R}^{d}} \mathbf{1}_{B}(z) \Psi_{t}\left(x, \frac{z-x}{\sigma} \right) p_{\sigma^{2} t}(z-x) m_{0}(x) dxdz.
\end{align*}
It follows that $\mathfrak{m}(t)$ possesses a density with respect to Lebesgue measure:
\[
	\frac{d\mathfrak{m}(t)}{d\lambda_{d}}(z) = m_{t}(z) \doteq \int_{\mathbb{R}^{d}} \Psi_{t}\left(x, \frac{z-x}{\sigma} \right) p_{\sigma^{2} t}(z-x) m_{0}(x) dx,\quad z\in \mathbb{R}^{d}.
\]
Thanks to \eqref{AppPsiEstimate}, the density is bounded:
\begin{align*}
	\|m_{t}\|_{\infty} &\leq C_{t}(\sigma,\|b\|_{\infty}, d) \cdot \sup_{z\in \mathbb{R}^{d}} \int_{\mathbb{R}^{d}} e^{\frac{\|b\|_{\infty}}{\sigma^2}|z-x|} p_{\sigma^{2} t}(z-x) m_{0}(x) dx \\
	&\leq C_{t}(\sigma,\|b\|_{\infty}, d) \cdot \|m_{0}\|_{\infty}\cdot  \int_{\mathbb{R}^{d}} e^{\frac{\|b\|_{\infty}}{\sigma^2}|x|} p_{\sigma^{2} t}(x) dx < \infty.
\end{align*}

It remains to prove inequality \eqref{AppBridgeInt}. Let $c\in [0,\infty)$. Since $W_{1},\ldots,W_{d}$ are independent one-dimensional Wiener processes, we have
\[
	\Mean\left[ \exp\left(c \sum_{i=1}^{d} \int_{0}^{t} \left|\int_{0}^{s} \frac{1}{t-r} dW_{i}(r) \right| ds \right) \right] = \Mean\left[ \exp\left(c \int_{0}^{t} \left|\int_{0}^{s} \frac{1}{t-r} dW_{1}(r) \right| ds \right) \right]^{d}.
\]
Let $\sgn\!: \mathbb{R}\rightarrow \{-1, 1\}$ denote the left-continuous version of the sign function, and denote by $\mathrm{Leb}_{t}$ Lebesgue measure on $[0,t]$. By \eqref{AppBridgeDifferential} and the definition of $\tilde{W}^{t,0}$, we have with $\Prb$-probability one,
\begin{equation} \label{AppBridgeSign}
	\int_{0}^{t} \left|\int_{0}^{s} \frac{1}{t-r} dW_{1}(r)\right| ds = -\int_{0}^{t} \sgn\left(\tilde{W}^{t,0}_{1}(s)\right) d\tilde{W}^{t,0}_{1}(s) + \int_{0}^{t} \sgn\left(\tilde{W}^{t,0}_{1}(s)\right) dW_{1}(s)
\end{equation}
since, for $\mathrm{Leb}_{t}\otimes \Prb$-almost all $(s,\omega)\in [0,t]\times \Omega$,
\begin{multline*}
	\left|\int_{0}^{s} \frac{1}{t-r} dW_{1}(r)\right| = \sgn\left(\int_{0}^{s} \frac{1}{t-r} dW_{1}(r) \right) \left( \int_{0}^{s} \frac{1}{t-r} dW_{1}(r)\right) \\
	= \sgn\left(\tilde{W}^{t,0}_{1}(s)\right) \left( \int_{0}^{s} \frac{1}{t-r} dW_{1}(r)\right).
\end{multline*}
Let $\tilde{L}^{t,0}(0)$ be the local time at the origin of the continuous semimartingale $\tilde{W}^{t,0}_{1}$ according to Theorem~3.7.1 in \cite[p.\,218]{karatzasshreve91}. In particular, $s\mapsto \tilde{L}^{t,0}_{s}(0)$ is a continuous non-decreasing non-negative process with $\tilde{L}^{t,0}_{0}(0) = 0$ such that, for every $s\in [0,t]$,
\[
	\lim_{\epsilon\searrow 0} \frac{1}{4\epsilon} \int_{0}^{s} \mathbf{1}_{[-\epsilon,\epsilon]}\left( \tilde{W}^{t,0}_{1}(r) \right)dr = \tilde{L}^{t,0}_{s}(0) \quad \Prb\text{-almost surely};
\]
cf.\ Theorem~3.7.1(iii) and Problem~3.7.6 in \cite{karatzasshreve91}. By the It{\^o}-Tanaka-Meyer formula applied to the absolute value function (see Eq.~(3.7.9) in 
\cite[p.\,220]{karatzasshreve91}), we have
\[
	-\int_{0}^{t} \sgn\left(\tilde{W}^{t,0}_{1}(s)\right) d\tilde{W}^{t,0}_{1}(s) = |\tilde{W}^{t,0}(0)| - |\tilde{W}^{t,0}(t)| + 2\tilde{L}^{t,0}_{t}(0) = 2\tilde{L}^{t,0}_{t}(0).
\]
This, together with \eqref{AppBridgeSign}, yields
\begin{multline*}
	\Mean\left[ \exp\left(c \sum_{i=1}^{d} \int_{0}^{t} \left|\int_{0}^{s} \frac{1}{t-r} dW_{i}(r) \right| ds \right) \right] \\
	= \Mean\left[ \exp\left( c\left( 2\tilde{L}^{t,0}_{t}(0) + \int_{0}^{t} \sgn\left(\tilde{W}^{t,0}_{1}(s)\right) dW_{1}(s) \right) \right) \right]^{d}.
\end{multline*}
Using again the Cauchy-Schwarz inequality and the (super-)martingale property of the stochastic exponential of a martingale, we find that
\begin{align*}
	& \Mean\left[ \exp\left(c \sum_{i=1}^{d} \int_{0}^{t} \left|\int_{0}^{s} \frac{1}{t-r} dW_{i}(r) \right| ds \right) \right] \\
	&\leq \Mean\left[ \exp\left( 4c\tilde{L}^{t,0}_{t}(0)\right) \right]^{d/2} \cdot \Mean\left[ \exp\left( 2c \int_{0}^{t} \sgn\left(\tilde{W}^{t,0}_{1}(s)\right) dW_{1}(s) - \frac{4c^2}{2}t + \frac{4c^2}{2}t \right) \right]^{d/2} \\
	&\leq \Mean\left[ \exp\left( 4c\tilde{L}^{t,0}_{t}(0)\right) \right]^{d/2} \cdot e^{c^{2}t\cdot d}.
\end{align*}
The distribution of $2\tilde{L}^{t,0}_{t}(0)$ is known explicitly. Let $(L_{s}(0))_{s\geq 0}$ be the local time at the origin of the one-dimensional Wiener process $W_{1}$. Then, as a consequence of \eqref{AppBMBridgeKernel}, we have that the distribution of $2\tilde{L}^{t,0}_{t}(0)$ coincides with the conditional distribution of $2L_{t}(0)$ given $W_{1}(t) = 0$. The joint distribution of $W_{1}(t)$ and $2L_{t}(0)$ is known to be absolutely continuous with respect to two-dimensional Lebesgue measure with density given by
\[
	\frac{\Prb\circ (W_{1}(t),2L_{t}(0))^{-1}}{d\lambda_{2}}(w,l) = \underbrace{ \mathbf{1}_{(0,\infty)}(l)\cdot \frac{l+|w|}{\sqrt{2\pi t^{3}}} \exp\left( -\frac{(l+|w|)^{2}}{2t} \right) }_{\doteq \phi(w,l)},\quad w,l\in \mathbb{R},
\]
see, for instance, Problem~6.3.4 in \cite[p.\,420]{karatzasshreve91}. Conditioning on $W_{1}(t) = 0$, we obtain
\[
	\frac{\Prb\circ (2\tilde{L}^{t,0}_{t}(0))^{-1}}{d\lambda_{1}}(l) = \frac{\phi(0,l)}{\int \phi(0,r)dr} = \mathbf{1}_{(0,\infty)}(l)\cdot \frac{l}{t} \exp\left( -\frac{l^{2}}{2t} \right),\quad l\in \mathbb{R}.
\]
We therefore have
\begin{multline*}
	\Mean\left[ \exp\left( 4c\tilde{L}^{t,0}_{t}(0)\right) \right] = \int_{0}^{\infty} \exp\left(2c\cdot l\right) \frac{l}{t} \exp\left( -\frac{l^{2}}{2t} \right) dl = 1 + 2c \int_{0}^{\infty} \exp\left(2c\cdot l - \frac{l^{2}}{2t} \right) dl \\
	= 1 + 2c\cdot e^{2c^{2} t} \int_{0}^{\infty} \exp\left(-\frac{(l - 2t c)^{2}}{2t} \right) dl \leq 1 + 2c\sqrt{2\pi t}\cdot e^{2c^{2} t}.
\end{multline*}
It follows that
\[
	\Mean\left[ \exp\left(c \sum_{i=1}^{d} \int_{0}^{t} \left|\int_{0}^{s} \frac{1}{t-r} dW_{i}(r) \right| ds \right) \right] \leq \left(1 + 2c\sqrt{2\pi t}\cdot e^{2c^{2} t} \right)^{d/2} \cdot e^{c^{2}t\cdot d},
\]
which proves \eqref{AppBridgeInt}.

\end{proof}

\end{appendix}

\subsubsection*{\bf Acknowledgements}
The first author presented the results on multiple solutions of Sections \ref{section pde} and \ref{two populations} at the INdAM workshop ``PDE models for multi-agent phenomena" in Rome on December 2nd, 2016. He received useful comments from Pierre Cardaliaguet and Diogo Gomes, in particular the suggestion of looking for results on short-time uniqueness in Lions' lectures. The first author is also grateful to Roberto Gianni for information about parabolic estimates and to Sara Farinelli for her careful critical reading of the first manuscript. Finally, we thank two anonymous referees for their insightful remarks.


\end{document}